\documentclass[preprint,12pt]{elsarticle}
\usepackage{amsfonts}
\usepackage{amsthm}
\usepackage{amsmath}
\usepackage[all]{xy}
\usepackage{titlesec}
\usepackage{dsfont}
\usepackage{amssymb}
\usepackage{extarrows}
\usepackage{mathrsfs}
\usepackage{arydshln}
\usepackage{multirow}
\usepackage{color}

\usepackage[bookmarksnumbered, colorlinks, plainpages]{hyperref}
in \textwidth 6 true in

\journal{xxx}
\newtheorem{theorem}{Theorem}[section]

\newtheorem{lemma}[theorem]{Lemma}
\newtheorem{corollary}[theorem]{Corollary}
\newtheorem{definition}[theorem]{Definition}
\newtheorem{remark}[theorem]{Remark}

\newcommand{\Z}{\ensuremath{\mathbb{Z}}}

\begin{document}

\begin{frontmatter}

\title{The relative James construction and its application to homotopy groups}

\author{Zhongjian Zhu}
\ead{20160118@wzu.edu.cn}
\author{Tian Jin}
\ead{lesuishenzhizhao@aliyun.com}
\address{College of Mathematics and Physics, Wenzhou University, Wenzhou 325035, China}

\begin{abstract}
In this paper, we develop the new method to compute the homotopy groups of the mapping cone $C_f=Y\cup_{f}CX$ beyond the metastable range by analysing the homotopy of  the $n$-th filtration of the relative James construction $J(X,A)$ for CW-pair $A\stackrel{i}\hookrightarrow X$, defined by B. Gray,  which is homotopy equivalent to the homotopy fiber of the pinch map $X\cup_{i}CA\rightarrow \Sigma A$. 
As an application,  we compute the 5 and 6-dim  unstable  homotopy groups of  3-dimensional  mod $2^r$ Moore spaces for all positive integers $r$.
\end{abstract}

\begin{keyword}
Relative James construction\sep higher order Whitehead product\sep homotopy group; Moore space

\MSC 55P10\sep55Q05\sep55Q15\sep 55Q52
\end{keyword}
\end{frontmatter}

\section{Introduction}
\label{intro}

Calculating the unstable homotopy groups of finite CW-complexes is a fundamental and difficult problem in algebraic topology. A lot of related work has been done on  mod $2^r$ Moore spaces when $r=1,2,3$. Let $P^{n}(p^r)$ denote  $n$-dimensional  elementary Moore spaces, whose only nontrivial reduced homology is  $\overline{H}_{n-1}(P^{n}(p^r))=\Z_{p^r}$ ($\Z_{k}:=\Z/k\Z$). J. Mukai computed  $\pi_{k}(P^{n}(2))$ for some $k\leq 3n-5$ and $3\leq n\leq 7$ in \cite{Mukai2,Mukai3}. In 1999, J.Mukai and T.Shinpo computed the  $\pi_{k} (P^{n}(4))$ in the range $k=2n-3,2n-4$ and $n\leq 24$ \cite{Mukai};  In 2007, X.G.Liu computed  $\pi_{k}(P^{n}(8)) $  in the range $k=2n-3$ for $2\leq n\leq 20$ and $k=2n-2$ for $3\leq n\leq 7$ \cite{X.G.Liu}. All the  homotopy groups are obtained above by applying  the  Theorem 2.1 of \cite{James I M} to the  homotopy exact sequences of a pair $(P^{n}(2^r), S^{n-1})$. These homotopy groups are metastable  (the homotopy group $\pi_{k}(X)$ of an $(n-1)$-connected CW-complex $X$ is called stable if $k\leq 2n-2$  and metastable if $k\leq 3n-2$. In the metastable range,  the exact EHP sequence holds). In 2001, J. Mukai computed the $\pi_{5}(P^{3}(2))$  and $\pi_{k}(P^{5}(2))$ with $10\leq k\leq 13$, which are beyond the metastable range. The main method is using the fiber sequence $\dots \Omega\Sigma X\rightarrow F\rightarrow Y\bigcup_{f}CX\xrightarrow{pinch} \Sigma X$, where $ Y\bigcup_{f}CX$ is the mapping cone of the map $f$, which is also denoted by $C_f$.
In 1972, B.Gray constructed the relative James construction $J(X,A)$ for a CW-pair $(X,A)$ and inclusion $A\stackrel{i}\hookrightarrow X$, which is filtrated by $J_{r}(X,A) (r=1,2,\dots)$  and proved that $J(X,A)$ is homotopy equivalent to the homotopy fiber of the pinch map $X\cup_{i}CA\rightarrow \Sigma A$ \cite{Gray}.
So the homotopy type of the above homotpy fiber  $F$ is homotopy equivalent to Gray's relative James construction $J(M_f,X)$ which is filtrated by a sequence of subspace $J_r(M_f,X), r=1,2,\dots$, where $M_f$ is the mapping  cylinder  of $f$.  When $X=\Sigma X'$, $Y=\Sigma Y'$, the  second filtration $J_2(M_f,X)$ of $J(M_f,X)$ has the homotopy type $Y\cup_{\gamma_2}C(Y \wedge X')$,   where $\gamma=[id_Y, f]$ is the generalized Whitehead product. This enables use to compute the homotopy groups of the mapping cones in metastable range \cite{JXYang, ZP A_3^2}. In general case, we don't know the homotopy type of the third filtration  $J_3(M_f,X)$ of $J(M_f,X)$. However, for the special cases $P^{k}(2)$, $k=3,5$, considered by Mukai,  the third filtration  $J_3(M_f,X)$  have the homotopy type $S^{k-1}\cup_{\gamma_2}e^{2k-2}\cup_{\gamma_{3}}e^{3k-3}$ and the informations of the attaching map $\gamma_3$ are given by I.M.James' note on cup products of this CW-complex (Theorem 3.3 of \cite{James cup-prod}), which provide the possibility to get the homotopy groups of these two Moore spaces beyond the metastable range. In 2003, Jie Wu also got a lot of unstable homotopy groups $\pi_{k}(P^{n}(2))(n\geq 3)$ by the functorial decomposition of the loop suspension space and the algebraic representation theory \cite{WJ Proj plane}.

In this paper,  we will get some information of homotopy type of the filtration $J_r(X,A)$ of the Gray's relative James construction $J(X,A)$, when $A=\Sigma A',X=\Sigma X'$ are suspensions.  Then $J_r(X,A)$ is the cofiber of the attaching map  $\gamma_r: \Sigma^{r-1}X'\wedge A'^{\wedge (r-1)}\rightarrow J_{r-1}(X,A)$.  We show that the  attaching map of $\gamma_r$ is an element of the set of $r$-th order Whitehead products defined by G.J.Porter \cite{Porter}. So we get the following Main Theorem for the homotopy of the homotopy fiber of the pinch map $ Y\bigcup_{f}CX\xrightarrow{pinch} \Sigma X$, where $X,Y$ are suspensions.

\begin{theorem}[\textbf{Main Theorem}]\label{thm 1.1}
	Let $X\xrightarrow{f}Y$ be a map of simply connected CW complexes $X=\Sigma X'$, $Y=\Sigma Y'$, then the homotopy fiber of the pinch map  $ C_{f}\xrightarrow{pinch} \Sigma X$
	has the homotopy type  
	\begin{align*}
		Y\cup_{\gamma_2}C(\Sigma Y'\wedge X')\cup_{\gamma_3}\cdots \cup_{\gamma_n}C(\Sigma^{n-1} Y'\wedge X'^{\wedge (n-1)})\cup_{\gamma_{n+1}}\cdots
	\end{align*}
	where $\gamma_r$ is an element in the set of $r$-th order Whitehead products  $[j^{r-1}_{Y}, j^{r-1}_{Y}f, \cdots, j^{r-1}_{Y}f]$ and $j^{r-1}_{Y}: Y\hookrightarrow Y\cup_{\gamma_2}C(\Sigma Y'\wedge X')\cup_{\gamma_3}\cdots \cup_{\gamma_{r-1}}C(\Sigma^{r-2} Y'\wedge X'^{\wedge (r-2)})$ is the canonical inclusion for $r=2,3,\dots$.
\end{theorem}

As an application,  we will use this to compute the homotopy groups $\pi_{k}(P^{3}(2^r))$ ($k=5,6$) for all positive integers $r$.

\begin{theorem}\label{thm1.2}
	\begin{align*}
		&\pi_{5}(P^{3}(2^r))\cong
		\left\{
		\begin{array}{ll}
			\Z_2\oplus\Z_2\oplus\Z_2 , & \hbox{$r=1$;} \\
			\Z_2\oplus\Z_2\oplus\Z_2\oplus\Z_{2^r}, & \hbox{$r\geq 2$.}
		\end{array}
		\right.  \\
		&\pi_6(P^{3}(2^r))\cong
		\left\{
		\begin{array}{ll}
			\Z_2\oplus \Z_2\oplus \Z_2\oplus \Z_2\oplus \Z_2, & \hbox{$r=1$;} \\
			\Z_2\oplus\Z_2\oplus \Z_4\oplus \Z_4\oplus \Z_{2^{r}} , & \hbox{$r\geq 2$.}
		\end{array}
		\right.
	\end{align*}
	
\end{theorem}
Recently, the research on the suspension homotopy of manifolds becomes a popular topic
\cite{So6dim,R.Huang 2,R.Huang Li,LiZhu,Theriault}, however the homology groups of  manifolds considered  by them   have no 2-torsion. And the known  unstable homotopy groups of mod $p^r$ ($p$ is odd prime) Moore spaces are very important in their methods. So the above results of unstable homotopy groups of mod $2^r$ Moore spaces have a potential application, that is to classify the homotopy types of the suspension of non-simpliy connected manifolds whose homology groups are allowed to have 2-torsion.

The paper is arranged as follows. Section \ref{sec: 2} introduces some concepts and properties of higher order Whitehead products;   In Section \ref{sec: 3} we introduce some concepts and lemmas about Relative James construction and  prove our Main Theorem \ref{thm 1.1} to give the methods to compute the unstable homotopy groups of a mapping cone.  In the last Section,  we apply the methods given in Section \ref{sec: 3} to compute the unstable homotopy groups of Moore spaces given in Theorem \ref{thm1.2}.

\section{Higher order Whitehead products }
\label{sec: 2}
In this paper,  all spaces and maps are in the category of pointed topological spaces and maps (i.e. continuous functions) preserving base point. And we always use $*$ and $0$ to denote the basepoint and the constant map mapping to the basepoint respectively. Without special mention, all spaces are CW-complexes and all the space pairs are CW-pairs. We denote $A\hookrightarrow X$  as an inclusion map.

By abuse of notion, we will not distinguish the notions between a map and its homotopy class in many cases.

Let $T_{r}(X_1,X_2,\dots,X_n)$ be the subset of the Cartesian product $X_1\times X_2\times\dots\times X_n$, consisting of those $n$-tuples with at least $r$ co-ordinates at a base point. Thus  $T_{n-1}(X_1,X_2,\dots,X_n)=X_1\vee X_2\vee\dots\vee X_n$, $T_{1}(X_1,X_2,\dots,X_n)$ is the  ``fat wedge", and  $X_1\times X_2\times\dots\times X_n/T_{1}(X_1,X_2,\dots,X_n)=X_1\wedge X_2\wedge\dots\wedge X_n$. From Theorem 1.2 and Theorem 2.1 of \cite{Porter},
there is a principle cofibration
\begin{align}
	\Sigma^{n-1} X_1\wedge \dots\wedge  X_n \xrightarrow{W_n} T_{1}(\Sigma X_1,\dots,\Sigma  X_n)\rightarrow \Sigma X_1\times \dots\times\Sigma X_n
	\label{Cof for Cartesian Pordu.}
\end{align}
where the map $W_n$ is natural.

Given a map $f: T_{1}(\Sigma X_1,\dots,\Sigma  X_n)\rightarrow X$, $n\geq 2$, define
\begin{align}
	W(f):=f_{\ast}(W_n)=fW_n\in [\Sigma^{n-1} X_1\wedge \dots\wedge  X_n, X]
	\nonumber
\end{align}
the $n$-th order Whitehead product, which depends only upon the homotopy class of $f$ \cite{Porter}.

Let $k_j:\Sigma X_j \hookrightarrow T_{1}(\Sigma X_1,\dots,\Sigma  X_n), j=1,2,\dots, n$, be the canonical injections.
\begin{definition}
	The set of $n$-th order Whitehead products of $f_j:\Sigma X_j\rightarrow X$, $j=1,\dots,n$, is
	\begin{align}
		[f_1,\dots, f_n]:=\{W(f)|f: T_{1}(\Sigma X_1,\dots,\Sigma  X_n)\rightarrow X, fk_j\simeq f_j, j=1,\dots,n\}.
		\nonumber
	\end{align}
\end{definition}
\begin{remark}\label{Remark of Def HOWP}
	1) $[f_1,\dots, f_n]$ is a subset of $[\Sigma^{n-1} X_1\wedge \dots\wedge  X_n, X]$, and it is perhaps empty. We also have
	$ [f_1,\dots, f_n]:=\{W(f):=fW_n~|~f: T_{1}(\Sigma X_1,\dots,\Sigma  X_n)\rightarrow X\}$ for all $f$ extending the wedge sum map $(f_1,\dots,f_n): \Sigma X_1\vee \dots\vee \Sigma X_n\rightarrow X$ up to homotopy.
	
	2)$[f_1,\dots, f_n]$ depends only upon the homotopy classes of $f_i$ $(i=1,\dots,n)$;
	
	3) Hardie\cite{Hardie} has also given the definition of $[f_1,\dots, f_n]$ when all  $X_i$ are spheres (called the $n$-th order spherical Whitehead product). $[f_1,f_2,f_3]$ is, in this case, the Zeeman product studied by Hardie\cite{Hardie Zeeman}. When $X_1, X_2$ are arbitrary, $[f_1,f_2]$ is the ``generalized Whitehead product" studied by Arkowitz\cite{Arkowitz}.
	
	4) From the Theorem 2.5 \cite{Porter},   if $X$ is an H-space, then $[f_1,\dots, f_n]=\{0 \}$; if $\Sigma$ is the suspension homomorphism, then $\Sigma [f_1,\dots, f_n]=\{0\}$.
\end{remark}

The following naturality of higher order Whitehead products comes from Theorem 2.1 of \cite{Porter}.
\begin{theorem}\label{Naturality HOWP}(Naturality)
	Let $f_i:A_i\rightarrow B_i$, $h_i:\Sigma B_i\rightarrow X$, $i=1,\dots,n$, $g:X\rightarrow Y$ and $\varphi: T_1(\Sigma B_1,\dots, \Sigma B_n)\rightarrow X$, then
	\begin{itemize}
		\item [(a)]$(\Sigma^{n-1}(f_1\wedge \dots\wedge f_n ))^{\ast}W(\varphi)=W(\varphi T_1(\Sigma f_1,\dots, \Sigma f_n))$
		\item [(b)] $g_{\ast}W(\varphi)=W(g\varphi)$
		\item [(c)] $(\Sigma^{n-1}(f_1\wedge \dots\wedge f_n ))^{\ast}[h_1,\dots,h_n]\subset [h_1\Sigma f_1,\dots,h_n \Sigma f_n]$
		\item [(d)]$g_\ast[h_1,\dots,h_n]\subset [gh_1,\dots,gh_n]$.
	\end{itemize}
\end{theorem}

\begin{corollary}\label{Naturality HOWP equivalence}
	$f_i$, $h_i$, $i=1,\dots,n$, $g$ are from the Theorem \ref{Naturality HOWP}
	\begin{itemize}
		\item [(a)]If $B_j=\Sigma B'_j$ for some $j$, then $k[h_1,\dots,h_j,\dots,h_n]\subset [h_1,\dots,kh_j,\dots h_n]$ for any integer $k$;
	\end{itemize}
	If $f_i:A_i\rightarrow B_i$ and  $g:X\rightarrow Y$ are homotopy equivalences, then
	\begin{itemize}
		\item [(b)] $(\Sigma^{n-1}(f_1\wedge \dots\wedge f_n ))^{\ast}[h_1,\dots,h_n]=[h_1\Sigma f_1,\dots,h_n \Sigma f_n]$;
		\item [(c)]$g_\ast[h_1,\dots,h_n]=[gh_1,\dots,gh_n]$.
	\end{itemize}
\end{corollary}
\begin{proof} For $(a)$,
	\begin{align}
		&[h_1,\dots,kh_j,\dots h_n]=[h_1\Sigma id_{B_1}, \dots,h_j\Sigma( k id_{B_j}) ,\dots, h_n \Sigma id_{B_n}] \nonumber \\
		\supset &(\Sigma^{n-1}id_{B_1}\wedge \dots \wedge  (k id_{B_j}) \wedge \dots \wedge id_{B_n} )^{\ast}[h_1,\dots,h_j,\dots h_n]\nonumber \\
		=&k [h_1,\dots,h_j,\dots h_n]. \nonumber
	\end{align}
	This conclusion can also be proved  by using Theorem 2.13 of \cite{Porter}.

	For $(b)$ and $(c)$, we only prove $(b)$, since the proof of $(c)$ is easier by the same method.
	Let $g_i$ be the homotopy inverse of $f_i$ $(i=1,2,\dots,n)$.
	\begin{align}
		&[h_1\Sigma f_1,\dots,h_n \Sigma f_n]=id^{\ast}_{\Sigma^{n-1}(A_1\wedge \dots\wedge A_n )}[h_1\Sigma f_1,\dots,h_n \Sigma f_n]   \nonumber \\
		=&(\Sigma^{n-1}(g_1\wedge \dots\wedge g_n )\Sigma^{n-1}(f_1\wedge \dots\wedge f_n ))^{\ast}[h_1\Sigma f_1,\dots,h_n \Sigma f_n] \nonumber \\
		=&(\Sigma^{n-1}(f_1\wedge \dots\wedge f_n ))^{\ast}(\Sigma^{n-1}(g_1\wedge \dots\wedge g_n ))^{\ast}[h_1\Sigma f_1,\dots,h_n \Sigma f_n] \nonumber \\
		\subset &(\Sigma^{n-1}(f_1\wedge \dots\wedge f_n ))^{\ast}[h_1\Sigma f_1\Sigma g_1,\dots,h_n \Sigma f_n\Sigma g_n] \nonumber \\
		=&(\Sigma^{n-1}(f_1\wedge \dots\wedge f_n ))^{\ast}[h_1,\dots,h_n ]. \nonumber
	\end{align}
	Then from  $(c)$ of Theorem 2.3, we get $(b)$ of this Corollary.
\end{proof}

The following lemma  about the property of third order Whitehead product is from the Theorem 4.2 of \cite{Golasinski deMelo}, and the special case for spheres $X_i (i=1,2,3)$ are given by Hardie \cite{Hardie Zeeman}.
\begin{lemma}\label{lemma for [a,b,c]}
	$[f_1,f_2,f_3]$ with $f_i:\Sigma X_i\rightarrow X$ $(i=1,2,3)$  is a coset of the following subgroup of  group $[\Sigma^2X_1\wedge X_2\wedge X_3, X]$
	\begin{align}
		[[\Sigma^2X_2\wedge X_3, X],  f_1]+[[\Sigma^2X_1\wedge X_3, X],  f_2]+[[\Sigma^2X_1\wedge X_2, X],  f_3]. \nonumber
	\end{align}
\end{lemma}

\section{Relative James construction}
\label{sec: 3}
Let  $(X,A)$ be a pair of  spaces with base point $*\in A$, and suppose that $A$ is closed in $X$. In \cite{Gray}, B.Gray construct a space $(X,A)_{\infty}$ analogous to the James construction, which is denoted by us as $J(X,A)$ to parallel with the the absolute James construction whose symbol is $J(X)$. In fact, $J(X,A)$ is the subset of $J(X)$ of words for which letters after the first are in $A$.  Especially,  $J(X,X)=J(X)$. As parallel with the familiar symbol $J_{r}(X)$ which is the $r$-th filtration of $J(X)$, we denote the  $r$-th filtration of $J(X,A)$ by  $J_{r}(X,A):=J(X,A)\cap J_{r}(X)$, which is denoted by Gray as  $(X, A)_r$ in  \cite{Gray}.
For example,
\begin{align}
	&J_{1}(X,A)=X, J_{2}(X,A)=(X\times A)/((a,\ast)\thicksim (\ast,a)), \nonumber\\
	&J_{3}(X,A)=(X\times A\times A )/((x,\ast, a)\thicksim  (x, a, \ast);  (a, a',\ast)\thicksim (\ast,a,a')\thicksim (a,\ast,a')).\nonumber
\end{align}
We denote an element in $J_{n}(X,A)$ by $\overline{(x,a_1,a_2,\dots, a_{n-1})}$. $J_{r-1}(X,A)$ is regarded as a subspace of $J_{r}(X,A)$ by the natural inclusion
\begin{align*}
	J_{r-1}(X,A)\xrightarrow{I_r} J_{r}(X,A), \quad \overline{(x,a_1,\dots, a_{r-1})}\mapsto \overline{(x,a_1,\dots, a_{r-1},\ast)}
\end{align*}
It is easy to see that $J_{n}(X,A)/J_{n-1}(X,A)$ is naturally homeomorphic to \newline $(X\times A^{n-1})/T_1(X,A,\dots, A)=X\wedge A^{\wedge (n-1)}$ and  there is a  pushout diagram for $r\geq 2$:
$$ \footnotesize{\xymatrix{
		X\times A^{r-1} \ar[r]^{\Pi_{r}} &  J_{r}(X, A) \\
		T_1(X,A,\dots,A) \ar@{^{(}->}[u]\ar[r]^{\nu'_{r}} & J_{r-1}(X, A)\ar@{^{(}->}[u]^{I_{r}} } }$$
\begin{align*}
	\text{ where}~~\Pi_{r}:	 & (x,a_1,\dots, a_{r-1})\mapsto \overline{(x,a_1,\dots, a_{r-1})}; \\
	\nu'_{r}: & (\ast,a_1,\dots, a_{r-1})\mapsto\overline{(a_1,\dots, a_{r-1})}   \\
	&	(x,a_1,\dots, a_{i-1},\ast, a_{i+1}\dots, a_{r-1})\mapsto\overline{(x,a_1,\dots, a_{i-1},a_{i+1},\dots, a_{r-1})}
\end{align*}
All the maps in the above pushout diagram are natural.

For  CW complex pair $(X,A)$ and inclusion $A\stackrel{i}\hookrightarrow X$,
Gray \cite{Gray} showed that the homotopy fiber $F_{p}$ of the pinch map   $ X\cup_{i}CA\xrightarrow{p}\Sigma A$ has the same homotopy type as $J(X,A)$ and $\Sigma J(X,A)\simeq \bigvee_{i\geq 0}(\Sigma X\wedge A^{\wedge i})$;  $\Sigma J_{k}(X,A)\simeq \bigvee_{i= 0}^{k-1}(\Sigma X\wedge A^{\wedge i})$; Moreover

\begin{theorem}\label{Gray's Thm J2(X,A)}[Gray,$1972$ ]
	For a CW complex pair $(X,A)$,   if $X=\Sigma X'$, $A=\Sigma A'$, then
	\begin{align}
		J_2(X,A)\simeq X\cup_{\gamma_2}C(X\wedge A'), \gamma_2=[id_X, i].  \nonumber
	\end{align}
\end{theorem}
Next we generalize above Gray's Theorem to the $J_n(X,A)$ for any $n\geq 2$ by the concepts of higher order Whitehead products.

Let $j^r_{X}:X=J_{1}(X,A)\hookrightarrow J_{r}(X,A)$ be the canonical inclusion.
\begin{theorem}\label{Thm Jn(X,A)}
	Let $(X,A)$ be a CW complex pair, $X=\Sigma X'$, $A=\Sigma A'$, and $A\stackrel{i_A}\hookrightarrow X$ be the inclusion. Then there is a principle cofibration sequence $\Sigma^{n-1}X'\wedge A'^{\wedge(n-1)}\xrightarrow{\gamma_n} J_{n-1}(X,A) \stackrel{I_n}\hookrightarrow J_n(X,A)$, where
	$\gamma_n$ is an element in the set of $n$-th order Whitehead products $ [j^{n-1}_{X}, j^{n-1}_Xi_{A},\dots,  j^{n-1}_Xi_{A}]$.

	Moreover, this principle cofibration sequence is natural for a map of the pairs  $(X_1', A_1')\rightarrow (X'_2, A_2')$.
\end{theorem}
\begin{proof}
	We have the following topological commutative diagram
	\begin{align}
		\small{\xymatrix{
				\Sigma^{n-1}X'\wedge A'^{\wedge(n-1)}  \ar@/^1pc/[rr]^-{W_n}\ar[r]_-{\bar{h}\simeq }&Q(X,A,\dots,A)\ar[r]_{\bar{\rho}}&T_1(X,A,\dots,A)\ar@{^{(}->}[r]\ar[d]^{\nu'_n}&X\times A^{n-1}\ar[d]^{\Pi_{n}}\\
				&Q(X,A,\dots,A)\ar@{=}[u]\ar[r]^{\nu'_n\bar{\rho}}& J_{n-1}(X,A)\ar@{^{(}->}[r]^{I_n} & J_{n}(X,A)
		} }. \label{diagram: Jn(X,A)}
	\end{align}
	where the space  $Q(X,A,\dots,A)$,  the natural map $\bar{\rho}$ and the natural homotopy equivalence  $\bar{h}$ are defined in Appendix of \cite{Porter}, which satisfy $\bar{\rho}\bar{h}=W_n$. Moreover, the right commutative square is a pushout diagram and
	$X\times A^{n-1}$ is naturally homeomorphic to the mapping cone $C_{\bar{\rho}}$ (Theorem 2.3 of \cite{Porter}).
	
	It is easy to see that $W_n\in [k_1, k_2,\dots, k_n]$. There are
	following  commutative  diagrams  for $2\leq r\leq n$
	\begin{align}
		\small{\xymatrix{
				A \ar@{^{(}->}[r]^-{k_r}\ar@{_{(}->}[d]_{i_A}&T_1(X,A,\dots,A)\ar[d]^{\nu'_n}&X\ar@{_{(}->}[l]_-{k_1} \ar[d]^{id_X}\\
				J_{1}(X,A)=X \ar@{^{(}->}[r]^-{j^{n-1}_{X}}&J_{n-1}(X,A) & J_{1}(X,A)=X\ar@{_{(}->}[l]_-{j^{n-1}_{X}}
		} }. \nonumber
	\end{align}
	i.e., $\nu'_{n}k_1=j^{n-1}_{X}$, $\nu'_{r}k_2=j^{n-1}_{X}i_A (r=2,\dots, n)$.
	
	So from $(d)$ of Theorem \ref{Naturality HOWP}, we have
	\begin{align}
		\gamma_n&:=\nu_n'W_n\in \nu'_{n\ast}[k_1, k_2,\dots, k_n]\subset [\nu'_{n}k_1, \nu'_{n}k_2,\dots, \nu'_{n}k_n]\nonumber \\
		&=[j^{n-1}_{X}, j^{n-1}_Xi_{A},\dots,  j^{n-1}_Xi_{A}]. \nonumber
	\end{align}
	By the Theorem 2.3 of \cite{Porter}, the top row is the principle cofibration sequence and the map $W_n$ is natural for a map of the pairs  $(X_1', A_1')\rightarrow (X'_2, A_2')$.  Then the naturality of $\gamma_n$ is coming from the naturality of the map $W_n$  and $\nu'_n$.
	
	Now  the conclusion that the bottom sequence in (\ref{diagram: Jn(X,A)}) is also a principle cofibration sequence will be obtained by the following Lemma.
\end{proof}

\begin{lemma}\label{Lem pushout}
	Suppose that the right square of the following commutative diagram is a pushout diagram
	\begin{align}
		\small{\xymatrix{
				A \ar[r]^-{f}&B\ar@{^{(}->}[r]^-{i}\ar[d]^{g}&C=B\cup_f CA\ar[d]^{j_C}\\
				A  \ar@{=}[u]\ar[r]^-{h=gf}& D\ar[r]^{j_D} & P
		} }. \nonumber
	\end{align}
	Then $P$ is homeomorphic to the mapping cone $C_h$.
	\begin{proof}
		\begin{align}
			&CA=([0,1]\times A)/(0,a)\thicksim \ast \thicksim (t,\ast) ;\nonumber \\
			&C=B\cup_f CA=(B\amalg CA)/(1,a)\thicksim f(a);\nonumber \\
			&P=(C\amalg D)/ i(b)\thicksim g(b)=(B\amalg CA\amalg D)/(1,a)\thicksim f(a), b\thicksim g(b); \nonumber\\
			&C_h=(D\amalg CA)/(1,a)\thicksim h(a). \nonumber
		\end{align}
		Denote  the elements in $P$  (resp. $C_h$)  by $\overline{x}$(resp. $\widehat{x}$)  for $x\in B\amalg CA\amalg D$  (resp. $x\in D\amalg CA$).  Define
		\begin{align}
			& \phi: P\rightarrow C_h, \overline{b}\mapsto \widehat{g(b)}; \overline{(t,a)}\mapsto  \widehat{(t,a)}; \overline{d}\mapsto \widehat{d} \nonumber \\
			& \varphi: C_h \rightarrow P, \widehat{d}\mapsto  \overline{d};  \widehat{(t,a)} \mapsto \overline{(t,a)}.\nonumber
		\end{align}
		It is easy to check $\phi, \varphi$ are well-defined maps and $\phi\varphi=id_{C_h}$, $\varphi\phi=id_{P}$.
	\end{proof}
\end{lemma}

Let $X\xrightarrow{f}Y$ be a map of CW complex pair.  We always use $F_f$ and $M_f$ to denote the homotopy fiber  and  mapping cylinder of $f$ respectively. There is  the following homotopy commutative  diagram  where $i_Y$ is an homotopy equivalence and  $r_Y$ is the homotopy inverse of $i_Y$ such that  $i_Yf\simeq i_X$ and $r_Yi_X=f$.
\begin{align}
	\footnotesize{\xymatrix{
			X\ar@{=}[d] \ar[r]^{f} & Y \ar@{^{(}->}[d]^{i_Y \simeq}\\
			X \ar@{^{(}->}[r]^{i_X} & M_f\ar@/^0.3pc/[u]^{r_Y} } } \label{homotopy Y to Mf}
\end{align}
Then we get $F_p\simeq J(M_f,X)$ where $p:C_f\rightarrow \Sigma X$ is the canonical pinch map. We  always denote the natural inclusions $Y\hookrightarrow M_f=J_1(M_f,X)\hookrightarrow J(M_f,X)\simeq F_p$  and $Y\hookrightarrow C_f$ by $j_p$ and $j_f$ respectively.

Now we can prove our Main Theorem (Theorem \ref{thm 1.1}) by  rewriting it as the following Theorem.
\begin{theorem}\label{Thm Jn(Mf,X)}
	Let $X\xrightarrow{f}Y$ be a map of simply connected CW complexes, $X=\Sigma X'$, $Y=\Sigma Y'$. Then 
	$J_n(M_f,X)$ has the homotopy type  $Y\cup_{\gamma_2}C(\Sigma Y'\wedge X')\cup_{\gamma_3}\cdots \cup_{\gamma_n}C(\Sigma^{n-1} Y'\wedge X'^{\wedge{n-1}})$, $\gamma_r$ is an element of $r$-th order Whitehead products in $[j^{r-1}_{Y}, j^{r-1}_{Y}f, \cdots, j^{r-1}_{Y}f]$ where $j^{r-1}_{Y}: Y\hookrightarrow Y\cup_{\gamma_2}C(\Sigma Y'\wedge X')\cup_{\gamma_3}\cdots \cup_{\gamma_{r-1}}C(\Sigma^{r-2} Y'\wedge X'^{\wedge{(r-2)}})$ is the canonical inclusion for $r=2,\dots n$.
\end{theorem}
\begin{proof}
	Define $J^{f}_{n}(Y,X)=(Y\times X^{n-1})/\thicksim$, where the relations are given by
	\begin{align}
		& (y,x_1,\dots,x_{i-1},\ast,x_{i+1},\dots, x_{n-1})\thicksim (y,x_1,\dots,x_{i-1},x_{i+1},\ast,x_{i+2},\dots, x_{n-1}) \nonumber \\
		&(\ast,x_1,x_2,\dots, x_{n-1})\thicksim (f(x_1),x_2,\dots,x_{i-1},\ast,x_{i+1},\dots, x_{n-1}) \nonumber
	\end{align}
	for any $i=1,\dots, n$ (when $i=1$, the symbol $x_{i-1}=x_0$ in the following means the first coordinate in $Y$).
	$J^{f}_{1}(Y,X)=Y$, $J^{f}_{2}(Y,X)=(Y\times X)/ (f(x),\ast)\thicksim (\ast, x)$.
	
	We denote the element in $J^{f}_{n}(Y,X)$ by $\overline{(y,x_1,\dots,x_{n-1})}^f$.
	There is a pushout diagram
	\begin{align}
		\small{\xymatrix{
				T_1(Y,X,\dots,X)\ar@{^{(}->}[r]\ar[d]^{\nu^{f}_n}&Y\times X^{n-1}\ar[d]^{\Pi^{f}_{n}}\\
				J^f_{n-1}(Y,X)\ar@{^{(}->}[r]^{I^{f}_n} & J^f_{n}(Y,X)
		} }. \nonumber
	\end{align}
	where the following maps are natural
	\begin{align}
		\Pi^{f}_{n}:& (y,x_1,\dots, x_{n-1})\mapsto\overline{(y,x_1,\dots, x_{n-1})}^f \nonumber \\
		\nu^{f}_{n}:& (\ast,x_1,\dots, x_{n-1})\mapsto\overline{(f(x_1),\dots, x_{n-1})}^f \nonumber\\
		&(y,x_1,\dots, x_{n-1},\ast, x_{i+1}\dots, x_{i-1})\mapsto\overline{(y,x_1,\dots, x_{i-1},x_{i+1},\dots, x_{n-1})}^f \nonumber\\
		I^{f}_{n}: &(y,x_1,\dots, x_{n-2})\mapsto\overline{(y,x_1,\dots, x_{n-2},\ast)}^f \nonumber
	\end{align}
	We will complete the proof by showing  (a): $J_{n}(M_f, X)\simeq J^{f}_{n}(Y,X)$; (b): there is a principle cofibration sequence $\Sigma^{n-1}Y'\wedge X'^{\wedge(n-1)}\xrightarrow{\gamma_n}J^{f}_{n-1}(Y,X)\hookrightarrow J^{f}_{n}(Y,X)$ with $\gamma_n\in [j^{r-1}_{Y}, j^{r-1}_{Y}f, \cdots, j^{r-1}_{Y}f]$.
	
	Proof of  $(a)$.
	
	Let $\theta_{n}:J_{n}(M_f, X)\rightarrow J^{f}_{n}(Y,X)$,
	$\overline{(m,x_1,\dots,x_{n-1})}:\mapsto\overline{(r_{Y}(m),x_1,\dots,x_{n-1})}^f$. where  $r_Y$ is the homotopy equivalence which comes from (\ref{homotopy Y to Mf}).
	It is well-defined  and we will show that $\theta_n$ are homotopy equivalences for every $n\geq 1$ by induction on $n$.
	
	There is a  topological commutative ladder of the cofibration sequences
	\begin{align}
		\small{\xymatrix{
				J_{n-1}(M_f,X) \ar[r]^-{I_n}\ar[d]^{\theta_{n-1}}&J_{n}(M_f,X)\ar[r]^{proj.}\ar[d]^{\theta_n}&M_f\wedge X^{\wedge n}\ar[d]^{r_Y\wedge id_{X}^{\wedge(n-1)}}\\
				J^f_{n-1}(Y,X) \ar[r]^-{I^f_n}& J^f_{n}(Y,X)\ar[r]^{proj.} & Y\wedge X^{n}
		} }. \nonumber
	\end{align}
	Since $\theta_1=r_Y: M_f\rightarrow Y$  and $r_Y\wedge id_{X}^{\wedge(n-1)}$ (for any $n\geq 1$) are homotopy equivalent,  $\theta_n$ induces isomorphisms on homologies.  $\theta_n$ is also homotopy equivalent since all the spaces considered are  simply connected.
	
	Proof of  $(b)$.
	
	We have the following topological commutative diagram with natural maps
	\begin{align}
		\small{\xymatrix{
				\Sigma^{n-1}Y'\wedge X'^{\wedge(n-1)}  \ar[r]^-{W_n}&T_1(Y,X,\dots,X)\ar@{^{(}->}[r]\ar[d]^{\nu^f_n}&Y\times X^{n-1}\ar[d]^{\Pi^f_{n}}\\
				\Sigma^{n-1}Y'\wedge X'^{\wedge(n-1)}  \ar@{=}[u]\ar[r]^-{\gamma_n}& J^f_{n-1}(Y,X)\ar@{^{(}->}[r]^{I^f_n} & J^f_{n}(Y,X)
		} }. \nonumber
	\end{align}
	where   
	\begin{align*}
		\gamma_n=\nu^f_nW_n\in \nu^f_n[k_1,k_2,\dots,k_n]\!\subset\! [\nu^f_nk_1,\nu^f_nk_2,\dots,\nu^f_nk_n]\!=\![j^{r-\!1}_{Y}, j^{r-\!1}_{Y}\!f, \cdots, j^{r-\!1}_{Y}\!f]
	\end{align*}

	Since the top sequence is the principle cofibration sequence, so is the bottom sequence by the the same proof of Theorem \ref{Thm Jn(X,A)}.
	
\end{proof}

\begin{remark}
	(1)For $n=2$, $J_{2}(M_f,X)\simeq Y\cup_{[id_Y, f]}C(Y\wedge X')$ by the above theorem. This also looks very obvious from Theorem \ref{Gray's Thm J2(X,A)}.  B. Gray, J. Mukai  and so on, they always use this result without proof.
\end{remark}
(2) The homotopy equivalence $\theta_n$ is natural. That is  the
commutative diagram left induces the commutative diagram right in the following

$ \footnotesize{\xymatrix{
		X\ar[d]^{\mu_X}\ar[r]^-{f} & Y\ar[d]^{\mu_Y}\\
		X_1\ar[r]^-{f_1} & Y_1}}$;~~~~~~$  \small{\xymatrix{
		J_{n}(M_f,X) \ar[r]^-{J_{n}(\widehat{\mu}, \mu_X)}\ar[d]^{\theta_{n}}&J_{n}(M_{f_1},X_1)\ar[d]^{\theta_n}\\
		J^f_{n}(Y,X) \ar[r]^-{J^f_{n}(\mu_Y, \mu_X)}& J^{f_1}_{n}(Y_1,X_1)
} }$
\newline
where the definition of the  map $J^f_{n}(\mu_Y, \mu_X)$ is canonical and $\widehat{\mu}$ is induced by $\mu_X$ and $\mu_Y$. $\gamma_n$ is also natural in the sense that the following diagram commutes for $\mu_X=\Sigma\mu'_X: X=\Sigma X'\rightarrow X_1=\Sigma X'_1$, $\mu_Y=\Sigma\mu'_Y: Y=\Sigma Y'\rightarrow Y_1=\Sigma Y'_1$.
\begin{align}
	\small{\xymatrix{
			\Sigma^{n-1}Y'\wedge X'^{\wedge(n-1)} \ar[r]^-{\gamma_n}\ar[d]^{\Sigma^{n-1}\mu'_Y\wedge \mu_X'^{\wedge(n-1)}}&J^f_{n-1}(Y,X)\ar[r]^{I^f_n}\ar[d]^{J^f_{n-1}(\mu_Y,\mu_X)}&J^f_{n}(Y,X)\ar[d]^{J^f_{n}(\mu_Y,\mu_X)}\\
			\Sigma^{n-1}Y_1'\wedge X_1'^{\wedge(n-1)}  \ar[r]^-{\gamma_n}& J^{f_1}_{n-1}(Y_1,X_1)\ar[r]^{I^f_n} & J^{f_1}_{n}(Y_1,X_1)
	} }. \nonumber
\end{align}

\begin{lemma}\label{J(X,A)toJ(X',A')}
	Suppose the left  diagram  is commutative\\
	$\footnotesize{\xymatrix{
			X\ar[d]^{\mu_X}\ar[r]^-{f} & Y\ar[d]^{\mu_Y}\\
			X_1\ar[r]^-{f_1} & Y_1}}$;~~~~~~$ \footnotesize{\xymatrix{
			F_p\simeq  J(M_f,X)\ar[d]^{J(\widehat{\mu},\mu_X)}\ar[r]&M_f/X\simeq C_f  \ar[d]^{\bar{\mu}}\ar[r]^-{p}& \Sigma X\ar[d]^{\Sigma\mu_X}\\
			F_{p_1}\simeq J(M_{f_1},X_1)\ar[r]& M_{f_1}/X_1 \simeq C_{f_1} \ar[r]^-{p_1}&\Sigma X'}}$
	
	then it induces the right commutative diagrams on fibrations, where $\widehat{\mu}$ satisfies
	
	$ \footnotesize{\xymatrix{
			&Y\ar@/^0.5pc/[rrr]^{\mu_Y}\ar@{^{(}->}[r]_{\simeq}& M_f \ar[r]_{\widehat{\mu}}& M_{f_1}\ar[r]_{\simeq}& Y_1} }$. 
	
	Let~~~$J_{r}(M_f,X)\xrightarrow{J(\widehat{\mu},\mu_X)|_{J_{r}(M_f,X)}=J_{r}(\widehat{\mu},\mu_X)} J_{r}(M_{f_1},X_1)$ $(r\geq 1)$,
	
	then we have the following commutative diagram
	$$ \footnotesize{\xymatrix{
			Y\wedge X^{\wedge ^{n-1}} \ar[d]_{\simeq} \ar[r]^{\mu_Y\wedge\mu_X^{\wedge ^{n-1}}} & Y_1\wedge X_1^{\wedge ^{n-1}} \ar[d]_{\simeq}\\
			J_{n}(M_f,X)/J_{n-1}(M_f,X)\ar[r]^-{\overline{J_{n}(\widehat{\mu},\mu_X)}} & J_{n}(M_{f_1},X_1)/J_{n-1}(M_{f_1},X_1)} }$$
\end{lemma}
\begin{proof}
	The above lemma is easily obtained from naturality of the relative construction $J(X,A)$ and $J_n(X,A)$.
\end{proof}

The following  lemma  comes from \cite[Lemma 2.3]{ZP A_3^2}, which generalizes the Lemma 4.4.1 of \cite{JXYang}.

\begin{lemma}\label{partial calculate1}
	Let $f: X\rightarrow Y$ be a map , then for the following fibration sequence,
	$$\Omega\Sigma X\xrightarrow{\partial} J(M_{f}, X)\simeq F_p \rightarrow C_f\xrightarrow{p} \Sigma X$$
	there exists homotopy-commutative diagram,
	$$ \footnotesize{\xymatrix{
			X\ar[d]_{\Omega\Sigma } \ar[r]^{f} & Y \ar@{_{(}->}[d]_{j_p}\\
			\Omega\Sigma X \ar[r]^-{\partial} & J(M_f, X)\simeq F_p } }$$
\end{lemma}

\begin{lemma}\label{lemm Suspen Fp}
	Let $X\xrightarrow{f}Y\stackrel{j_f}\hookrightarrow C_f\xrightarrow{p}\Sigma X$ be a cofibration sequence of CW complexes. Consider the fibration sequence  $\Omega\Sigma^{k+1}X\xrightarrow{\partial_k}F_{\Sigma^kp} \xrightarrow{i_k}\Sigma^k C_f\xrightarrow{\Sigma^kp}\Sigma^{k+1}X$.

	If $i_{k\ast}:[\Sigma^k Y,F_{\Sigma^k p}]\rightarrow [\Sigma^kY, \Sigma^kC_f]$ is bijective, then there is a map $\phi_k$ such that
	
	1) the following diagram is homotopy commutative:

	$$ \footnotesize{\xymatrix{
			Y \ar@{^{(}->}[r]^-{j_p}\ar@{^{(}->}[d]_{\Omega^k\Sigma^k } & F_p \ar[d]^{\phi_k}\\
			\Omega^k\Sigma^k Y \ar[r]^-{\Omega^kj_{\Sigma^kp}} & \Omega^kF_{\Sigma^kp} } }$$
\end{lemma}

2) the following diagram with rows fibration sequence is homotopy commutative:

\begin{align}
	\small{\xymatrix{
			\Omega\Sigma X \ar[r]^-{\partial}\ar[d]^{\Omega(\Omega^{k}\Sigma^k)}&F_p\ar[r]^{i}\ar[d]^{\phi_k}& C_f\ar[d]^{\Omega^{k}\Sigma^k}\ar[r]^{p}& \Sigma X\ar[d]^{\Omega^{k}\Sigma^k}\\
			\Omega^{k+1}\Sigma^{k+1}X  \ar[r]^-{\Omega^{k}\partial_k}&\Omega^{k}F_{\Sigma^kp} \ar[r]^{\Omega^{k}i_{k}} &\Omega^{k}\Sigma^k C_f \ar[r]^-{\Omega^{k}\Sigma^k p}& \Omega^{k}\Sigma^{k+1} X
	} }. \nonumber
\end{align}
$\Omega^k\Sigma^k$ represents the $k$-fold canonical inclusions.

\begin{proof}
	Consider the following  left diagram where the bottom row is a fibration sequence
	
	$\small{\xymatrix{
			\Sigma^{k}F_p \ar[r]_-{\Sigma^{k}i}&\Sigma^{k}C_f\ar[r]_{\Sigma^{k}p}\ar[d]^{id}&\Sigma^{k+1}X\ar[d]^{id}\\
			F_{\Sigma^{k}p} \ar[r]^-{i_k}& \Sigma^{k}C_f\ar[r]^{\Sigma^{k}p} & \Sigma^{k+1}X
	} }.$ ~~~~~$\small{\xymatrix{
			\Sigma^{k}Y\ar@/^0.7pc/[rr]^{\Sigma^{k}j_f}\ar[r]_-{\Sigma^k j_p}\ar@{=}[d]&\Sigma^{k}F_p \ar[r]_-{\Sigma^{k}i}\ar[d]^{T}&\Sigma^{k}C_f\ar[d]^{id}\\
			\Sigma^{k}Y\ar@/_0.7pc/[rr]_{\Sigma^{k}j_f}\ar[r]^-{j_{\Sigma^kp}}&F_{\Sigma^{k}p} \ar[r]^-{i_k}& \Sigma^{k}C_f
	} }. $
	
	By $\Sigma^{k}p\Sigma^{k}i=0$, there exists a map $T:\Sigma^{k}F_{p}\rightarrow F_{\Sigma^{k}p}$ such that the left square of the first diagram (i.e., the right square of the second diagram) is homotopy commutative.
	
	By the second diagram, we get $i_kT\Sigma^k j_p=i_kj_{\Sigma^kp}$, then we get $T\Sigma^k j_p=j_{\Sigma^kp}$ since  $i_{k\ast}:[\Sigma^k Y,F_{\Sigma^k p}]\rightarrow [\Sigma^kY, \Sigma^kC_f]$ is bijective.

	Then one has the following  diagrams:
	\begin{align} \small{\xymatrix{
				Y \ar[d]^{\Omega^k\Sigma^k}\ar[r]^-{j_p}& F_p\ar[d]^{\Omega^k\Sigma^k}\\
				\Omega^k\Sigma^{k}Y \ar[d]^{id}\ar[r]^-{\Omega^k\Sigma^{k}j_p}&\Omega^k\Sigma^{k}F_p\ar[d]^{\Omega^k T}\\
				\Omega^k\Sigma^{k}Y \ar[r]^-{\Omega^kj_{\Sigma^kp}}& \Omega^kF_{\Sigma^{k}p}
		} }. \nonumber
	\end{align}
	The top square is commutative by the fact that the functor $\Omega^k$ is the adjoint of $\Sigma^k$
	and the bottom square is homotopy commutative by the naturality of the functor $\Omega^k$. The
	map $\phi_k$ is the composition of two right vertical maps. This two squares give the required
	commutative diagram in 1).
	
	Similarly the another commutative square involving the map $\phi_k$ gives the commutativity of the middle square in 2).
	
	The commutativity of the left square in 2) follows from the fact that the connecting
	map of $\Omega^k F_{\Sigma^{k}p}\xrightarrow{\Omega^ki_k} \Omega^k\Sigma^{k}C_f\xrightarrow{\Omega^k \Sigma^kp}\Omega^k \Sigma^{k+1}X$ is $\Omega^k$ of the connecting map of $F_{\Sigma^kp} \xrightarrow{i_k}\Sigma^k C_f\xrightarrow{\Sigma^kp}\Sigma^{k+1}X$.
\end{proof}

Now we give the  steps to compute the homotopy group $\pi_{k}$ of the mapping cone $C_f$ for  a map $X\xrightarrow{f}Y$, where $X,Y$ are suspensions
\begin{enumerate}
	\item [\textbf{Step 1:}]	Consider the fibration sequence  $\Omega\Sigma X\xrightarrow{\partial} F_p\rightarrow C_f\xrightarrow{p}\Sigma X$, where $F_p\simeq J(M_f,X)$ is the 
	homotopy fiber of the pinch map $p$ and analysis the homotopy type of the skeleton  $Sk_{m}(F_p)\simeq J_r(M_f,X)~(m>k)$;
	\item [\textbf{Step 2:}]	Compute the cokernel $Coker (\partial_{k})_{\ast}$ and kernel  $Ker(\partial_{k-1})_{\ast}$ in the following induced exact sequence 
	\begin{align*}
		\pi_{k+1}(X)\xrightarrow{(\partial_{k})_{\ast}}\pi_{k}(F_p)\rightarrow \pi_{k}(C_f)\xrightarrow{p_{\ast}} \pi_{k}(X)\xrightarrow{(\partial_{k-1})_{\ast}} \pi_{k-1}(F_p)
	\end{align*}
	where $\pi_{k}(F_p)=\pi_{k}(Sk_{m}(F_p))\cong \pi_{k}(J_r(M_f,X))$;
	\item [\textbf{Step 3:}]	Determine the group structure of $\pi_k(C_f)$ from the following short exact sequence 
	\begin{align*}
		0\rightarrow  Coker (\partial_{k})_{\ast}\rightarrow \pi_{k}(C_f)\rightarrow Ker(\partial_{k-1})_{\ast}\rightarrow  0.
	\end{align*}
\end{enumerate}

\section{Application: compute $\pi_k(P^{3}(2^r)), k=5,6$}
\label{sec: 4}

In this section we compute the homotopy groups of the Elementary Moore spaces  $\pi_{k}(P^{3}(2^r)), k=5,6$ for all $r\geq 1$ under 2-localization. We should note that $\pi_{5}(P^{3}(2))\cong \Z_2\oplus\Z_2\oplus\Z_2$ is  given by J. Mukai \cite{Mukai4} and J. Wu \cite{WJ Proj plane}.

The following generators of homotopy groups of spheres after localization at 2  come from \cite{Toda}.
$\iota_n=[id]\in\pi_n(S^n)$; $\pi_{3}(S^2)=\Z_{(2)}\{\eta_2\}$;  $\pi_{n+1}(S^n)=\Z_2\{\eta_n\} (n\geq 3)$;
$\pi_{n+2}(S^n)=\Z_2\{\eta_n^2\} (n\geq 3)$; $\pi_{5}(S^2)=\Z_2\{\eta_2^3\}$; $\pi_{6}(S^3)=\Z_4\{\nu'\}$;  $\pi_{7}(S^4)=\Z_4\{\Sigma\nu'\}\oplus \Z_{(2)}\{\nu_4\}$; $\pi_{6}(S^2)=\Z_4\{\eta_2\nu'\}$;  $\pi_{7}(S^3)=\Z_2\{\nu'\eta_6\}$, where $\Z_{(2)}$ denotes the $2$-local integers.

Let $j^{n}_1$ and $j^m_2$ be the inclusions of $S^{n}$ and  $S^m$ respectively into $S^{n}\vee S^m$;   $q^{n}_1$ and  $q^m_2$ be the projections from $S^{n}\vee S^m$  to  $S^{n}$ and $S^m$ respectively.

Denote the map $f:S^n\vee S^m\rightarrow  X$ satisfying $fj_1^n=f_1$ and $fj_2^m=f_2$ by $(f_1, f_2)$ and the map $j_1^n f q_1^n+j_2^m g q_2^m: S^n\vee S^m\rightarrow  S^n\vee S^m$  by $f\vee g$.

There is a canonic cofibration sequence for mod $2^r$ Moore space.
\begin{align}
	S^2\xrightarrow{2^{r}\iota_2} S^{2}\xrightarrow{i_{r}} P^3(2^r)\xrightarrow{p_{r}}  S^{3} \label{Cofiberation for Mr^k}
\end{align}
Let $\Omega S^{3}\xrightarrow{\partial_r}  F_{p_{r}}\xrightarrow{i_{p_r}} P^{3}(2^r)\xrightarrow{p_{r}}  S^{3}$ be the homotpy fiber sequence.
By Theorem \ref{Thm Jn(Mf,X)},  we get $Sk_{6}(F_{p_{r}})=Sk_{7}(F_{p_{r}})\simeq J_{3}(M_{2^r\iota_2}, S^2)\simeq S^2\cup_{\gamma_2}CS^3\cup_{\gamma_3}CS^5$.

$\gamma_2=[\iota_2, 2^r\iota_2]=2^r[\iota_2, \iota_2]$. Since  $[\iota_2, \iota_2]=2\eta_2$, $\gamma_2=[\iota_2, 2^r\iota_2]= 2^{r+1}\eta_2$. Thus 
$Sk_4(F_{p_{r}})\simeq S^2\cup_{2^{r+1}\eta_2}CS^3$.

Let $L^4_m=S^2\cup_{2^m\eta_2}CS^3 (m\geq 1)$, then $L^4_0=\mathbb{C}P^2$ and $Sk_4(F_{p_{r}})\simeq L^4_{r+1}$. Let
$j^L_{m}: S^2\hookrightarrow L^4_m$ be the canonical inclusion. Then 
\begin{align*}
	\gamma_3\in [j^L_{r+1}, j^L_{r+1}(2^{r}\iota_2). j^L_{r+1}(2^r\iota_2)]\subset \pi_5(L_{r+1}^4).
\end{align*}
In order to simplify the notion, sometimes the inclusions $S^n\hookrightarrow Sk_k(F_p)$ and $S^n\hookrightarrow Sk_k(F_p)\hookrightarrow F_p$, we will use the same symbol.
\subsection{Generators of  $\pi_{k}(L_m^4)$, $k=5,6$}

There is a cofibration sequence $S^3\xrightarrow{2^{m}\eta_2}S^2 \stackrel{j^L_{m}}\hookrightarrow L^4_{m}\xrightarrow{p^{L}_{m}}S^4$ and the following fibration sequence
$\Omega S^4\xrightarrow{\partial^L_{m}} F_{p^{L}_{m}}\xrightarrow{\tau_{m}^L}  L^4_{m}\xrightarrow{p^{L}_{m}}S^4$. $Sk_6(F_{p^{L}_{m}})\simeq J_{2}(M_{2^{m}\eta_2}, S^3)=S^2\cup_{[\iota_2,2^m\eta_2]}CS^4=S^2\vee S^5$, since $[\iota_2,\eta_2]=0$ (p.76 of \cite{Toda Whitehead Prod.} ). Hence there are  inclusions $j_{m}^{F}: S^2\vee S^5\simeq  Sk_{6}(F_{p^{L}_{m}})\hookrightarrow F_{p^{L}_{m}}$, 
$j_{m}^{S^5}=j_{m}^{F}j_2^5: S^5\hookrightarrow  F_{p^{L}_{m}}$ and 
$j_{p_m^L}=j_{m}^{F}j_1^2: S^2\hookrightarrow  F_{p^{L}_{m}}$.

\begin{lemma}\label{lem:pi5(Lm4)}
	\begin{align}
		\pi_5(L_m^4)=  \left\{
		\begin{array}{ll}
			\Z_{(2)}\{\beta_0\}, & \hbox{$m=0$;} \\
			\Z_{(2)}\{\beta_1\}\oplus\Z_4\{\tilde{\eta}_4\}, & \hbox{$m=1$;} \\
			\Z_{(2)}\{\beta_m\}\oplus\Z_2\{j_{m}^L\eta_2^3\}\oplus\Z_2\{\tilde{\eta}_4\}, & \hbox{$m>1$.}
		\end{array}
		\right. \label{pi5(Lm)}
	\end{align}
	where $\beta_m=\tau_m^Lj_m^{S^5}$, $\tilde{\eta}_4$ is a lift of $\eta_4$, i.e., $p_{m\ast}^{L}\tilde{\eta}_4=\eta_4$.
\end{lemma}
\begin{proof}
	From the Lemma \ref{partial calculate1}, we have the following exact sequence with commutative squares
	\begin{align} \small{\xymatrix{
				\pi_{6}( S^4)\ar[r]^-{(\partial^L_{m})_{5\ast}}&\pi_{5}(F_{p^{L}_{m}})\ar[r]&\pi_{5}(L_m^4)\ar[r]&\pi_{5}(S^4)\ar[r]^-{ (\partial^L_{m})_{4\ast}}& \pi_{4}(F_{p^{L}_{m}})\\
				\pi_{5}(S^3)\ar[r]^-{(2^m\eta_2)_{\ast}}\ar[u]_{\Sigma \cong } &\pi_{5}(S^2)\ar[u]_{ j_{p_m^L\ast}}&&\pi_{4}( S^3)\ar[r]^-{(2^m\eta_2)_{\ast}}\ar[u]_{\cong\Sigma } &\pi_{4}(S^2)\ar[u]_{j_{p_m^L\ast}} }}   \nonumber
	\end{align}
	
	The left and right commutative squares above imply
	\begin{align}
		(\partial^L_{m})_{4\ast}, (\partial^L_{m})_{5\ast}~\text{are}~0 ~\text{for}~m\geq 1;~\text{isomorphic}  ~\text{for}~m=0 \label{partial4,5 on Lm4}
	\end{align}
	
	Thus it easy to get $\pi_5(L_{0}^4)=\pi_5(\mathbb{C}P^2)=\Z_{(2)}\{\tau_0^Lj_0^{S^5}\}$.
	
	For $m\geq 1$, we get the following short exact sequence with the left commutative triangle
	\begin{align} \small{\xymatrix{
				0\ar[r]&\pi_{5}(F_{p^{L}_{m}})\ar[r]^{\tau^{L}_{m\ast}}&\pi_{5}(L_m^4)\ar[r]^{p_{m\ast}^{L}}&\pi_{5}(S^4)\ar[r]& 0 \\
				&\pi_{5}(S^2\vee S^5)=\Z_2\{j_1^2\eta_2^3\}\oplus\Z_{(2)}\{j_2^5\iota_5\}\ar[ur]_-{~~~~~(j_{m\ast}^L,~ (\tau_m^Lj_m^{S^5})_\ast)}\ar@{^{(}->}[u]^{\cong }&&\Z_2\{\eta_4\}\ar@{=}[u]&}}   \nonumber
	\end{align}
	From Proposition 2.13 of \cite{Yamaguchi}, we get this lemma.
\end{proof}

Next we compute $\pi_{6}(L_m^4)$. There is the exact sequence
\begin{align}
	\small{\xymatrix{
			\Z_{(2)}\{\nu_4\}\!\oplus\!\Z_4\{\Sigma\nu'\}\!=\!\pi_{7}( S^4)\!\ar[r]^-{(\partial^L_{m})_{6\ast}}&\!\pi_{6}(F_{p^{L}_{m}})\!\ar[r]^{\tau_{m\ast}^L}&\!\pi_{6}(L_m^4)\!\ar[r]^{p_{m\ast}^L}&\!\pi_{6}(S^4)\!\ar[r]^-{ (\partial^L_{m})_{5\ast}}& \!\pi_{5}(F_{p^{L}_{m}}) }} \label{long exact for pi6(Lm4)}
\end{align}
where
$\pi_6(F^L_{p_{m}})=\Z_4\{j^L_{p_m}\eta_2\nu'\}\oplus \Z_2\{j^{S^5}_{m}\eta_5\}\oplus \Z_{(2)}\{j_m^F[j_1^2, j_2^5]\}\cong \pi_6(J_2(M_{2^m\eta_3},S^3))=\pi_6(S^2\vee S^5)$.

\begin{lemma}\label{lem:pi6(Lm4)}
	\small{\begin{align}
			&\pi_{6}(L_m^4)\!=\!\left\{
			\begin{array}{ll}
				\Z_2\{\beta_1\eta_5\}\!\oplus \!\Z_4\{J^{LF}_1[j_1^2, j_2^5]\}\!\oplus\!\Z_2\{\bar{\lambda}\tilde{\eta}_{4}^2\}, & \! \! \hbox{$m=1$;} \\
				\Z_2\{\beta_2\eta_5\}\!\oplus \!\Z_2\{j^L_{2}\eta_2\nu'\!+\!2J^{LF}_2[j_1^2, j_2^5]\}\!\oplus\!\Z_8\{J^{LF}_2[j_1^2, j_2^5]\}\!\oplus\!\Z_2\{\bar{\lambda}\tilde{\eta}_{4}^2\}, & \!\!\hbox{$m=2$;} \\
				\Z_2\{\beta_m\eta_5\}\!\oplus\! \Z_4\{j^L_{m}\eta_2\nu'\}\!\oplus\!\Z_{2^m}\{J_{m}^{LF}[j_1^2, j_2^5]\}\!\oplus\!\Z_2\{\bar{\lambda}\tilde{\eta}_{4}^2\}, &\! \!\hbox{$m\geq 3$.}
			\end{array}
			\right.
			\nonumber
	\end{align}}
	where $J_{m}^{LF}=\tau_{m}^Lj_m^F$, $\bar{\lambda}\tilde{\eta}_{4}^2$ is a lift of $\eta^2_4\in\pi_6(S^4)$, i.e., $p_{m\ast}^{L}(\bar{\lambda}\tilde{\eta}_{4}^2)=\eta^2_4$.
\end{lemma}
\begin{proof}
	For $m\geq 0$, assume that
	\begin{align}
		(\partial^L_{m})_{6\ast}(\nu_4)=a_m j^L_{p_m}\eta_2\nu'+b_mj^{S^5}_{m}\eta_5+c_mj_m^F[j_1^2, j_2^5], a_m\in \Z_4, b_m\in \Z_2, c_m\in \Z \label{Assum for partial6nu4}
	\end{align}

	We have the following  two commutative diagrams for cofibration sequences left and fibration sequences right respectively.
	\begin{align}
		\small{\xymatrix{
				S^3 \ar@{=}[d]\ar[r]^-{2^m\iota_3}&S^3\ar[d]_{\eta_2}\ar[r]^{i_m} & P^{4}(2^m) \ar[d]_{\bar{\lambda}}\ar[r]^{p_m}& S^4\ar@{=}[d]\\
				S^3 \ar[r]^-{2^m\eta_2}&S^2 \ar[r]^{j^L_m}& L_m^4\ar[r]^{p_m^L}& S^4
		} }
		, ~~\small{\xymatrix{
				\Omega S^4 \ar@{=}[d]\ar[r]^-{\partial_{m} }&F_{p_{m}}\ar[d]_{\lambda}\ar[r]^{i_{p_m}} & P^{4}(2^m) \ar[d]_{\bar{\lambda}}\ar[r]^{p_m}& S^4\\
				\Omega S^4 \ar[r]^-{\partial_m^L }&F_{p_{m}^L} \ar[r]^{\tau_m^L}& L_m^4\ar[r]^{p^L_m}& S^4 \ar@{=}[u]
		} }. \label{diagram 1 M2m4 to Lm4}
	\end{align}
	Then from Lemma \ref{J(X,A)toJ(X',A')}, there is the following commutative diagrams
	\begin{align}
		\small{\xymatrix{
				\pi_7(S^4) \ar@{=}[d]\ar[r]^-{(\partial_{m})_{6\ast} }&\pi_6(F_{p_{m}})\ar[d]_{\lambda_{\ast}} \\
				\pi_7(S^4) \ar[r]^-{(\partial_m^L)_{6\ast} }&\pi_6(F_{p_{m}^L})
		} };
		\small{\xymatrix{
				\pi_6(S^3)\ar[d]_{\eta_{2\ast}}\ar[r]^-{j_{p_m\ast}}&\pi_6(F_{p_{m}})\cong \pi_6(J_2(M_{2^m\iota_3},S^3))\ar[d]_{\lambda_{\ast}}\ar[r]^-{p_{\iota\ast}^6} & \pi_6(S^3\wedge S^3)\ar[d]_{(\eta_{2}\wedge \iota_3)_{\ast}}\\
				\pi_6(S^2)\ar[r]^-{j_{p_m\ast}^L}&\pi_6(F_{p_{m}^L})\cong \pi_6(J_2(M_{2^m\eta_2},S^3))\ar[r]^-{p_{\eta\ast}^5}  & \pi_6(S^3\wedge S^2)
		} } \label{diagram M2m4 to Lm4}
	\end{align}

	where $\pi_6(F_{p_{m}})=\Z_4\{j_{p_m}\nu'\}\oplus \Z_{(2)}\{j_{p_m}^6\}\cong \pi_6(J_2(M_{2^m\iota_3},S^3))=\pi_6(S^3\vee S^6)$,  $j^6_{p_{m}}: S^6\hookrightarrow Sk_7(F_{p_{m}})\simeq S^3\vee S^6$ is the inclusion of the wedge summand $ S^6$ (Section 3.1 of \cite{ZP A_3^2}); $p_{\iota}^6$ and $p_{\eta}^5$ are the quotient maps $J_2(M_{f},S^3)\rightarrow J_2(M_{f},S^3)/J_1(M_{f},S^3)$  for $f=2^m\iota_3$ and $f=2^m\eta_2$ respectively.
	
	From the right commutative square of the second diagram of (\ref{diagram M2m4 to Lm4}), we can assume
	\begin{align}
		\lambda_{\ast}(j_{p_m}^{6})=x_mj_{p_m}^L\eta_2\nu'+j^{S^5}_{m}\eta_5+ y_mj_m^F[j_1^2, j_2^5]~~\text{for some integers}~x_m,y_m.  \label{assum lambda jpm6}
	\end{align}
	By the proof of (3.3), (3.5) of \cite{ZP A_3^2}, for $m\geq 0$
	we get
	\begin{align}
		(\partial_m)_{6\ast}(\Sigma\nu')= 2^{m}j_{p_m}\nu',~~~~~  (\partial_m)_{6\ast}(\nu_4)=\pm 2^{m-1}j_{p_m}\nu'+2^mj_{p_m}^6 \label{partialm on pi7S4}
	\end{align}
	where   $\pm 2^{0-1}=t_0$ also represents some integer  in (\ref{partialm on pi7S4}) for $m=0$.
	
	Thus by the first commutative square of  (\ref{diagram M2m4 to Lm4}), we have 
	\begin{align}
		&(\partial_{m}^{L})_{6\ast}(\Sigma\nu') =\lambda_{\ast}(\partial_m)_{6\ast}(\Sigma\nu')=\lambda_{\ast}( 2^{m}j_{p_m}\nu')=2^mj_{p_m}^L\eta_2\nu' \label{partial6(Sigmanu')Lm4}\\
		&(\partial_{m}^{L})_{6\ast}(\nu_4) =\lambda_{\ast}(\partial_m)_{6\ast}(\nu_4)=\lambda_{\ast}(\pm 2^{m-1}j_{p_m}\nu'+2^mj_{p_m}^6) \nonumber\\
		&=(\pm2^{m-1}+2^m x_m)j_{p_m}^L\eta_2\nu'+2^mj^{S^5}_{m}\eta_5+2^my_mj_m^F[j_1^2, j_2^5].\nonumber
	\end{align}
	Comparing above equation with (\ref{Assum for partial6nu4}) to get $c_m=2^my_m$ and \begin{align}
		\!\!\!\!\!\!\small{\begin{tabular}{r|c|c|c|}
				\cline{2-4}
				& $ m=1$&$m=2$&$m\geq 3$ \\
				\cline{2-4}
				$\Z_4\!\ni\! a_m\!=\!\pm2^{m-1}\!+\!2^m x_m\!=\!$  &  $  \pm  1$ & $2$ &$0$\\
				\cline{2-4}
		\end{tabular}};~~b_m\!=\!\left\{
		\begin{array}{ll}
			1, &\! \hbox{$m=0$;} \\
			0, &\! \hbox{$m=1$.}
		\end{array}
		\right.  \label{equation am bm}
	\end{align}
	
	It is well known that $\Omega \mathbb{C}P^2\simeq S^1\times \Omega S^5$, implies that $\pi_6(\mathbb{C}P^2)\cong \Z_2$.
	Hence by the exact sequence (\ref{long exact for pi6(Lm4)}) for $m=0$ and (\ref{partial4,5 on Lm4}),(\ref{assum lambda jpm6}), (\ref{partial6(Sigmanu')Lm4}), (\ref{equation am bm}) we get \footnote{In equation $(\ref{equi c0})$, here we get $y_0=\pm 1$  without doing $2$-localizing. But under $2$-localization, we only get $y_0$ is odd.}
	\begin{align}
		&\Z_2\cong \pi_6(\mathbb{C}P^2)=\pi_6(L_0^4)\cong Coker (\partial_{0}^L)_{6\ast}\nonumber\\
		&=\frac{\Z_4\{j^L_{p_0}\eta_2\nu'\}\oplus \Z_2\{j^{S^5}_{0}\eta_5\}\oplus \Z_{(2)}\{j_0^F[j_1^2, j_2^5]\}}{\left \langle j^L_{p_0}\eta_2\nu', a_0j^L_{p_0}\eta_2\nu'+j^{S^5}_{0}\eta_5+y_0j_0^F[j_1^2, j_2^5]\right \rangle} \cong \frac{\Z_2\oplus \Z_{(2)}}{\left \langle (1, y_0)\right \rangle} \nonumber\\
		\Rightarrow &~~~(\partial_{0}^L)_{6\ast}(\nu_4)=a_0j^L_{p_0}\eta_2\nu'+j^{S^5}_{0}\eta_5+c_0j_0^F[j_1^2, j_2^5],~~ c_0=y_0=\pm 1~\label{equi c0}
	\end{align}

	In order to get $c_m (m\geq 1)$, consider the following commutative diagrams
	\begin{align}
		\small{\xymatrix{
				S^3 \ar[d]^{2^m\iota_3}\ar[r]^-{2^m\eta_3}&S^2\ar@{=}[d]\ar[r]^{j^L_m} &  L_m^4 \ar[d]_{\bar{\chi}}\ar[r]^{p_m^L}& S^4\ar[d]_{2^m\iota_4}\\
				S^3 \ar[r]^-{\eta_2}&S^2 \ar[r]^-{j^L_0}& L_0^4=\mathbb{C}P^2\ar[r]^-{p_0^L}& S^4
		} }
		, ~~ \small{\xymatrix{
				\Omega S^4\ar[d]^{\Omega (2^m\iota_4)}\ar[r]^-{\partial_m^L}&F_{p_{m}^L}\ar[d]_{\chi}\ar[r]^{\tau_m^L} & L_{m}^4 \ar[d]_{\bar{\chi}}\ar[r]^{p^L_m}& S^4\ar[d]_{2^m\iota_4}\\
				\Omega S^4 \ar[r]^-{\partial_0^L }&F_{p_{0}^L} \ar[r]^{\tau_0^L}& L_0^4\ar[r]^{p^L_0}& S^4
		} }\label{diagrams 1 for Lm4 to CP2}
	\end{align}
	The following left homotopy commutative diagram and right commutative diagram are induced by the above diagrams from Lemma \ref{J(X,A)toJ(X',A')}
	\begin{align}
		\footnotesize{\xymatrix{
				S^2 \ar@{=}[d]\ar[r]^-{j_{p_m^L}}&J_{2}(M_{2^m\eta_2},\! S^3)\!=\!S^2\!\vee \!S^5\ar[d]^{\chi |_{J_2}}\ar[r]^-{p_m^5} &  S^2\!\wedge \!S^3 \ar[d]_{\iota_2\wedge 2^m\iota_3}\\
				S^2 \ar[r]^-{j_{p_0^L}}&J_{2}(M_{\eta_2}, S^3)\!=\!S^2\!\vee\! S^5\ar[r]^-{p_0^5} &  S^2\!\wedge \!S^3
		} };~~\footnotesize{\xymatrix{
				\pi_7(S^4)\ar[d]^{(2^m\iota_4)_{\ast}}\ar[r]^-{(\partial_m^L)_{6\ast}}&\pi_6(F_{p_{m}^L})\ar[d]_{\chi_{\ast}} & \pi_6(S^2\!\vee \!S^5)\ar[l]_{j_{m\ast}^F}^{\cong}\ar[d]_{(\chi |_{J_2})_{\ast}}\\
				\pi_7( S^4) \ar[r]^-{(\partial_0^L)_{6\ast} }&\pi_6(F_{p_{0}^L}) & \pi_6(S^2\!\vee \!S^5)\ar[l]_{j_{0\ast}^F}^{\cong}
		} } \label{diagrams for Lm4 to CP2}
	\end{align}
	where $\chi|_{J_2}$ is the restriction of $\chi: F_{p_{m}^L}\simeq J(M_{2^m\eta_2},S^3)\rightarrow  F_{p_{0}^L}\simeq J(M_{\eta_2},S^3)$ and from the left diagram above we get
	\begin{align}
		\chi |_{J_2}=\iota_2\vee (\pm 2^m\iota_5)+\varepsilon j_1^2\eta^3_2 q_2^5, \varepsilon\in \Z_2.  \label{chiJ2}
	\end{align}
	
	Let  $P_3: \pi_6(S^2\vee S^5)=\Z_4\{j^2_{1}\eta_2\nu'\}\oplus \Z_2\{j^{5}_{5}\eta_5\}\oplus \Z_{(2)}\{[j_1^2, j_2^5]\}\rightarrow \Z_{(2)}\{[j_1^2, j_2^5]\} $ be the projection to the last summand.
	\begin{align}
		&P_3\chi_{\ast}(j_{m\ast}^{F})^{-1}(\partial_m^L)_{6\ast}(\nu_4)=  P_3\chi_{\ast}(a_mj_1^2\eta_2\nu'+c_m[j_1^2,j_2^5]) ~(\text{by}~(\ref{Assum for partial6nu4}), (\ref{equation am bm}))\nonumber\\
		&=P_3(\iota_2\vee (\pm 2^m\iota_5)+\varepsilon j_1^2\eta^3_2 q_2^5)_{\ast}(c_m[j_1^2,j_2^5]) ~(\text{by}~(\ref{chiJ2}))\nonumber\\
		&=P_3(c_m[j_1^2, \pm 2^m j_2^5 ]+ c_m[0,\varepsilon j_1^2\eta^3_2]) =\pm 2^m c_m [j_1^2,  j_2^5].\nonumber
	\end{align}
	On the other hand, by the right commutative diagram in (\ref{diagrams for Lm4 to CP2})
	\begin{align}
		&P_3\chi_{\ast}(j_{m\ast}^{F})^{-1}(\partial_m^L)_{6\ast}(\nu_4)=P_3 (j_{0\ast}^{F})^{-1}(\partial_0^L)_{6\ast}(2^m\iota_4)_{\ast}(\nu_4) \nonumber\\
		&=P_3 (j_{0\ast}^{F})^{-1}(\partial_0^L)_{6\ast}(2^{2m}\nu_4) ~(\text{by Lemma A.1 of \cite{ZP A_3^2}})\nonumber\\
		&=P_3 (j_{0\ast}^{F})^{-1}(2^{2m}(a_0j^L_{p_0}\eta_2\nu'+j^{S^5}_{0}\eta_5+c_0j_0^F[j_1^2, j_2^5])) =2^{2m} c_0 [j_1^2,  j_2^5]. \nonumber\\
		\Rightarrow &~~~~\pm 2^m c_m=2^{2m} c_0=\pm 2^{2m}~~~~\Rightarrow~~~ c_m=\pm 2^m. \label{equ for cm}
	\end{align}
	From (\ref{partial4,5 on Lm4}), (\ref{Assum for partial6nu4}),(\ref{partial6(Sigmanu')Lm4}), (\ref{equation am bm}), (\ref{equ for cm}), we get the following short exact sequence
	\begin{align}
		0\rightarrow Coker(\partial_{m}^L)_{6\ast}\xrightarrow{\tau_{m\ast}^L} \pi_{6}(L_m^4)\xrightarrow{p_{m\ast}^L}\pi_{6}(S^4)\rightarrow 0~~~~~(m\geq 1) \label{Short exact pi6Lm4}
	\end{align}
	\begin{align}
		\text{where}~~ &Coker(\partial_{m}^L)_{6\ast}=\frac{\Z_4\{j^L_{p_m}\eta_2\nu'\}\oplus \Z_2\{j^{S^5}_{m}\eta_5\}\oplus \Z_{(2)}\{j_m^F[j_1^2, j_2^5]\}}{  \left \langle 2^m j^L_{p_m}\eta_2\nu', \pm 2^{m-1}j^L_{p_m}\eta_2\nu'\pm2^mj_m^F[j_1^2,j_2^5] \right \rangle }\nonumber\\
		&=\left\{
		\begin{array}{ll}
			\Z_2\{j^{S^5}_{1}\eta_5\}\oplus \Z_4\{j_1^F[j_1^2, j_2^5]\}, & \hbox{$m=1$;} \\
			\Z_2\{j^{S^5}_{2}\eta_5\}\oplus \Z_2\{j^L_{p_2}\eta_2\nu'\!+\!2j_2^F[j_1^2, j_2^5]\}\oplus\Z_8\{j_2^F[j_1^2, j_2^5]\}, & \hbox{$m=2$;} \\
			\Z_2\{j^{S^5}_{m}\eta_5\}\oplus \Z_4\{j^L_{p_m}\eta_2\nu'\}\oplus\Z_{2^m}\{j_m^F[j_1^2, j_2^5]\}, & \hbox{$m\geq 3$.}
		\end{array}
		\right.
		\nonumber
	\end{align}
	From Section 3.1 of \cite{ZP A_3^2}, there is an element $\tilde{\eta}_4^2\in \pi_6(P^{4}(2^m))$ with order 2  such that $p_{m}(\tilde{\eta}_4^2)=\eta_4^2$. By the commutative diagram (\ref{diagram 1 M2m4 to Lm4}), $\bar{\lambda}\tilde{\eta}_4^2$ is an order 2 lift of $\eta_4^2\in \pi_{6}(S^4)$ which implies that the short exact sequence (\ref{Short exact pi6Lm4}) splits. We complete the proof of Lemma \ref{lem:pi6(Lm4)}.
\end{proof}

\subsection{The attaching map $\gamma_3:S^5\rightarrow L_{r+1}^4$}

\begin{lemma}\label{lemma Sigma Tpi5(Lm4)} For the suspension homorphism $  \pi_5(L_m^4)\xrightarrow{\Sigma} \pi_6(\Sigma L_m^4)=\pi_6(S^3\vee S^5)=\Z_4\{j_1^3\nu'\}\oplus\Z_2\{j_2^5\eta_5\}$ $(m\geq 2)$, we have
	\begin{align}
		\Sigma(j_{m}^L\eta_2^3)=2j_1^3\nu'; ~~~~\Sigma\tilde{\eta}_4=j_2^5\eta_5+2\varepsilon j_1^3\nu', \varepsilon\in \{0,1\}. \nonumber
	\end{align}
\end{lemma}
\begin{proof}For $m\geq 2$, we have the fibration sequence  $F_{\Sigma p_m^L}\rightarrow \Sigma L_{m}^4=S^3\vee S^5\xrightarrow{\Sigma p_m^L=p_2^5} S^5$ with $Sk_7F_{\Sigma p_m^L}=S^3\vee S^7$.
	From the Lemma \ref{lemm Suspen Fp}, there is a map $\phi_1: F_{p_m^L}\rightarrow \Omega F_{\Sigma p_m^L}$ such that the following diagram commutative:
	\begin{align}
		\small{\xymatrix{
				\pi_5(S^2) \ar@{_{(}->}[r]_-{j_{p_{m}^L\ast}}\ar[d]^{\Sigma}\ar@/^1pc/[rr]^{j^L_{m\ast}}&\pi_5(F_{p_{m}^L})\ar[r]_{\tau_{m\ast}^L}\ar[d]^{\phi_{1\ast}}& \pi_5(L_m^4)\ar[d]^{\Sigma}\ar[r]_{p_{m\ast}^L}& \pi_5(S^4)\ar[d]^{\Sigma~\cong}\\
				\pi_6(S^3)\ar@{^{(}->}[r]^-{j_{\Sigma p_{m}^L\ast}\cong}\ar@/_1pc/[rr]_{j^3_{1\ast}}&\pi_6(F_{\Sigma p_{m}^L}) \ar@{^{(}->}[r] & \pi_6(S^3\vee S^5)=\Z_4\{j_1^3\nu'\}\oplus\Z_2\{j_2^5\eta_5\}  \ar[r]^-{ p_{2\ast}^5}& \pi_6(S^5)
		} }. \nonumber
	\end{align}
	Thus we get the results of this lemma from the above commutative diagram.
\end{proof}

\begin{lemma}\label{lemma lambda3}
	$\gamma_3=\pm3\cdot 2^r\beta_{r+1} \in \pi_5(L_{r+1}^4)$ for  $r\geq 1$.
\end{lemma}
\begin{proof}
	Assume that $\gamma_3=aj_{r+1}^L\eta_2^3+b\beta_{r+1}+c\tilde{\eta}_4\in  \pi_5(L_{r+1}^4)$ for some $a,c\in \Z_2$, $b\in \Z$.
	
	The following commutative diagram  is induced by (\ref{diagrams 1 for Lm4 to CP2})
	
	$
	\small{\xymatrix{
			\pi_5(L_m^4) \ar[d]^{\bar{\chi}_\ast}&\pi_5(F_{p_{m}^L})\ar@{_{(}->}[l]\ar[d]^{\chi_\ast}& \pi_5(J_2(M_{2^m\eta_2},S^3))=\pi_5(S^2\vee S^5)\ar@{_{(}->}[l]^-{\cong}\ar[d]^{(\chi|_{J_2})_{\ast}}\ar[r]^-{proj}&\pi_5(S^5)\ar[d]_{(2^m\iota_5)_{\ast}}\\
			\pi_5(\mathbb{C}P^2)&\pi_5(F_{p_{0}^L})\ar[l]& \pi_5(J_2(M_{\eta_2},S^3))=\pi_5(S^2\vee S^5)\ar@{_{(}->}[l]^-{\cong}\ar[r]^-{proj}&\pi_5(S^5)
	} }; $
	
	since $\pi_5(\mathbb{C}P^2)=\Z_{(2)}\{\beta_0\}$ is a torsion free group, take $m=r+1$ above to get
	\begin{align}
		\bar{\chi}_\ast(\gamma_3)=\bar{\chi}_\ast(aj_{r+1}^L\eta_2^3+b\beta_{r+1}+c\tilde{\eta}_4)=b\bar{\chi}_\ast(\beta_{r+1})=\pm2^{r+1}\beta_0. \label{kappa(beta)}
	\end{align}
	On the other hand,
	\begin{align}
		&\bar{\chi}_\ast(\gamma_3)\in \bar{\chi}_\ast[j^L_{r+1}, j^L_{r+1}(2^r\iota_2),j^L_{r+1}(2^r\iota_2)]\subset  [\bar{\chi}j^L_{r+1}, \bar{\chi}j^L_{r+1}(2^r\iota_2),\bar{\chi}j^L_{r+1}(2^r\iota_2)] \nonumber\\
		=&   [j^L_{0}, j^L_{0}(2^r\iota_2),j^L_{0}(2^r\iota_2)]\supset  2^{2r} [j^L_{0}, j^L_{0},j^L_{0}]~~~~~~~~~~\text{( (a) of Corollary \ref{Naturality HOWP equivalence})} \nonumber
	\end{align}
	From Corollary 2 of \cite{Porter Postnikov}, the set $[j^L_{0}, j^L_{0},j^L_{0}]$ has only one element $-6\beta_0$ or $6\beta_0\in \pi_5(\mathbb{C}P^2)$.
	
	Thus$-6\cdot  2^{2r} \beta_0$ or $6\cdot  2^{2r} \beta_0 \in [j^L_{0}, j^L_{0}(2^r\iota_2),j^L_{0}(2^r\iota_2)]$.
	
	By Lemma \ref{lemma for [a,b,c]}, $[j^L_{0}, j^L_{0}(2^r\iota_2),j^L_{0}(2^r\iota_2)]$ is a coset of subgroup
	\begin{align}
		[\pi_4(\mathbb{C}P^2), j^L_{0}(2^r\iota_2)]+[\pi_4(\mathbb{C}P^2), j^L_{0}(2^r\iota_2)]+[\pi_4(\mathbb{C}P^2), j^L_{0}\iota_2] \nonumber
	\end{align}
	which is zero since $\pi_4(\mathbb{C}P^2)=0$ (Lemma 2.4 of \cite{Yamaguchi}). Thus
	\begin{align}
		&[j^L_{0}, j^L_{0}(2^r\iota_2),j^L_{0}(2^r\iota_2)]=\{6\cdot  2^{2r} \beta_0 \} ~\text{or}~\{-6\cdot  2^{2r} \beta_0 \} \nonumber\\
		\Rightarrow&~~~~~~~ \bar{\chi}_\ast(\gamma_3)=\pm   6\cdot  2^{2r} \beta_0\nonumber.
	\end{align}
	Comparing with (\ref{kappa(beta)}), we get
	\begin{align}
		b=\pm 3\cdot 2^r\nonumber.
	\end{align}
	From 4) of Remark \ref{Remark of Def HOWP} and Lemma \ref{lemma Sigma Tpi5(Lm4)}
	\begin{align}
		&0=\Sigma \gamma_3=2aj_1^3\nu'\pm 3\cdot 2^r\Sigma(\beta_{r+1})+c(j_2^5\eta_5+2\varepsilon j_1^3\nu') \nonumber\\
		&=(2a+2c)j_1^3\nu'\pm 3\cdot 2^r\Sigma(\beta_{r+1})+cj_2^5\eta_5\nonumber.
	\end{align}
	Note that: for $r\geq 2$, $3\cdot 2^r\Sigma(\beta_{r+1})=0$;
	for $r=1$, by Lemma 2.3 of \cite{Mukai4}, $6\beta_2=6\tau_3$, where the definition of $\tau_3$ comes from Lemma 2.2 of \cite{Mukai4}.  So $\Sigma 6\beta_2=0$ from the proof of Lemma 3.3 of \cite{Mukai4}.
	
	Thus $a,c=0$ for $r\geq 1$.
\end{proof}

\subsection{Calculation of $\pi_{k}(P^{3}(2^r)), k=5,6$}
\begin{theorem}\label{thmpi5(M3)}
	$\pi_{5}(P^{3}(2^r))\cong
	\left\{
	\begin{array}{ll}
		\Z_2\oplus\Z_2\oplus\Z_2 , & \hbox{$r=1$;} \\
		\Z_2\oplus\Z_2\oplus\Z_2\oplus\Z_{2^r}, & \hbox{$r\geq 2$.}
	\end{array}
	\right.
	$
\end{theorem}
\begin{proof}
	There is a cofibration sequence
	\begin{align}
		S^5\xrightarrow{\gamma_3}L_{r+1}^4\simeq J_{2}(M_{2^r\iota_2}, S^2) \xrightarrow{I_3} J_{3}(M_{2^r\iota_2}, S^2)\xrightarrow{p_{J_3}^M} S^6, \label{Cof of J3M2r}
	\end{align}
	which induces the following exact sequence by Theorem 1.16 of \cite{Cohen}
	\begin{align}
		\pi_5(S^5)\xrightarrow{\gamma_{3\ast}}\pi_5(L_{r+1}^4)\rightarrow \pi_5(J_{3}(M_{2^r\iota_2}, S^2))\rightarrow 0\nonumber.
	\end{align}
	From the Lemma \ref{lemma Sigma Tpi5(Lm4)} and Lemma \ref{lemma lambda3}, it is easy to get (under 2-localization)
	\begin{align}
		\pi_5(F_{p_r})\cong \pi_5(J_{3}(M_{2^r\iota_2},S^2))=\Z_{2^r}\{I_3\beta_{r+1}\}\oplus\Z_{2}\{I_3j_{r+1}^L\eta_2^3\}\oplus\Z_{2}\{ I_3\tilde{\eta}_4\}\nonumber.
	\end{align}
	We have the following  two commutative diagrams for cofibration sequences left and fibration sequences right respectively for $s\leq r$
	\begin{align}
		\small{\xymatrix{
				S^2 \ar@{=}[d]\ar[r]^-{2^s\iota_2}&S^2\ar[d]_-{2^{r-\!s}\iota_2}\ar[r]^-{i_s} & P^{3}(2^s) \ar[d]_-{\bar{\psi}_{r}^{s}}\ar[r]^-{p_s}& S^3\ar@{=}[d]\\
				S^2 \ar[r]^-{2^r\iota_2}&S^2 \ar[r]^-{i_r}& P^{3}(2^r)\ar[r]^-{p_r}& S^3
		} }, ~~\small{\xymatrix{
				\Omega S^3 \ar@{=}[d]\ar[r]^-{\partial_s }&F_{p_{s}}\ar[d]_-{\psi_{r}^{s}}\ar[r]^-{i_{p_s}} &P^{3}(2^s) \ar[d]_-{\bar{\psi}_{r}^{s}}\ar[r]^-{p_s}& S^3\\
				\Omega S^3 \ar[r]^-{\partial_r }&F_{p_{r}} \ar[r]^-{i_{p_r}}& P^{3}(2^r)\ar[r]^-{p_r}& S^3 \ar@{=}[u]
		} }
		. \label{diagram M2s to M2r}
	\end{align}
	Then the above right diagram  induces the following  commutative diagram of exact sequences for $s=1$
	\begin{align}
		\footnotesize{\xymatrix{
				\pi_6( S^3) \ar@{=}[d]\ar[r]^-{(\partial_1)_{5\ast} }&\pi_5(F_{p_{1}})\ar[d]_{\psi_{r\ast}^1}\ar[r]^{i_{p_1\ast}} & \pi_5(P^{3}(2) )\ar[d]_{\bar{\psi}_{r\ast}^1}\ar[r]^{p_{1\ast}}& \pi_{5}(S^3)\ar[r]^-{(\partial_1)_{4\ast} }&\pi_4(F_{p_{1}})\ar[d]_{\psi_{r\ast}^1}\\
				\pi_{6}( S^3 )\ar[r]^-{(\partial_r)_{5\ast} }&\pi_5(F_{p_{r}})\ar[r]^{i_{p_r\ast}}& \pi_5(P^{3}(2^r)) \ar[r]^{p_{r\ast}}& \pi_{5}(S^3) \ar[r]^-{(\partial_r)_{4\ast} }\ar@{=}[u] &\pi_4(F_{p_{r}})
		} } \label{diam p5 M1 to Mr}
	\end{align}
	Using Lemma \ref{partial calculate1}, it is easy to get $Ker (\partial_{r})_{4\ast}=\pi_5(S^3)= \Z_2\{\eta_3^2\}$ for $r\geq 1$.
	
	Lemma \ref{J(X,A)toJ(X',A')} implies 
	\begin{align*}
		\psi_{r}^1 j_{p_1}=j_{p_r}2^{r-1}\iota_2=2^{r-1}j_{p_r}: \footnotesize{\xymatrix{S^2\ar@/_1pc/[rr]_{2^{r-1}j_{p_r}}\ar@{^{(}->}[r]^{j_{p_1}} &  F_{p_{1}}\ar[r]^{\psi_{r}^1} &F_{p_{r}} }}
	\end{align*}
	and from the proof of Theorem 1.1 of \cite{Mukai4}, we get  $(\partial_1)_{5\ast}(\nu')=j_{p_1}\eta_{2}^3$. So
	\begin{align}
		(\partial_r)_{5\ast}(\nu')\!=\!\psi_{r\ast}^1(\partial_1)_{5\ast}(\nu')\!=\!\psi_{r\ast}^1 (j_{p_1}\eta_{2}^3)\!=\! j_{p_r\ast}(2^{r-1}\iota_2)_\ast(\eta_{2}^3)\!=\!0 ~(r\geq 2)\label{partial5 on nu'}
	\end{align}
	which implies $Coker (\partial_{r})_{5\ast}=\pi_5(F_{p_{r}})$. Hence there is the  following short exact sequence
	\begin{align}
		\small{\xymatrix{
				0\ar[r]&\pi_5(F_{p_{r}})\ar[r]^{i_{p_r\ast}}& \pi_5(P^{3}(2^r)) \ar[r]^-{p_{r\ast}}& \pi_{5}(S^3)=\Z_2\{\eta_2^3\} \ar[r] &0
		} } \label{exact of Pi5(Mr)}
	\end{align}

	Since  $\pi_{5}(P^{3}(2))\cong \Z_2\oplus\Z_2\oplus\Z_2$, there is an element $\xi_1\in \pi_{5}(P^{3}(2))$ with order 2 which is a lift of  $\eta_2^3$. Thus by the commutative diagram (\ref{diam p5 M1 to Mr}), $\bar{\psi}_{r}^1\xi_1\in \pi_{5}(P^{3}(2^r))$ is also a lift  of  $\eta_2^3$ with order 2. Thus the short exact sequence (\ref{exact of Pi5(Mr)}) splits.
	
	Hence  $\pi_5(P^{3}(2^r))\cong \pi_5(F_{p_{r}})\oplus \pi_{5}(S^3)\cong \Z_2\oplus\Z_2\oplus\Z_2\oplus\Z_{2^r}$  for $r\geq 2$.
\end{proof}

\begin{lemma}\label{lem:pi6(J3M2r)}
	$\pi_6(F_{p_{r}})\cong \pi_6(J_3(M_{2^r\iota_2},S^3))\cong
	\left\{
	\begin{array}{ll}
		\Z_2\oplus \Z_2\oplus \Z_2\oplus \Z_2, & \hbox{$r=1$;} \\
		\Z_2\oplus \Z_2\oplus \Z_4\oplus \Z_{2^{r}} , & \hbox{$r\geq 2$.}
	\end{array}
	\right.
	$
\end{lemma}
\begin{proof}
	By (\ref{Cof of J3M2r}), we have  fibration sequence:
	\begin{align}
		\Omega S^6\xrightarrow{\partial_{J_3}^{M}}F_{J_3}^{M}\rightarrow J_3(M_{2^r\iota_2},S^3)\xrightarrow{p_{J_3}^M} S^6 \label{fiber J3M2r}
	\end{align}
	where $F_{J_3}^{M}$ is the homotopy fiber of $J_3(M_{2^r\iota_2},S^3)\xrightarrow{p_{J_3}^M} S^6$, with $Sk_{9}(F_{J_3}^{M})\simeq J_2(M_{\gamma_3}, S^5) $ $\simeq L_{r+1}^4\cup_{\alpha_r}C(L_{r+1}^4\wedge S^4)$. So there is a cofibration sequence
	\begin{align}
		S^6\simeq S^2\wedge S^4 \xrightarrow{\alpha_r(j_{r+1}^L\wedge \iota_4)} J_1(M_{\gamma_3}, S^5)=L_{r+1}^4 \rightarrow Sk_7(F_{J_3}^{M})\rightarrow S^7. \label{Cofiber Sk7FJ3M}
	\end{align}
	We have the following  two commutative diagrams for cofibration sequences left and fibration sequences right respectively.
	
	$  \small{\xymatrix{
			S^5 \ar@{=}[d]\ar[r]^-{\gamma_3}&L_{r+1}^4\ar[r] & J_3(M_{2^r\iota_2},\!S^3)\ar[r]^-{p_{J_3}^M}& S^6\ar@{=}[d]\\
			S^5\ar[r]^-{\pm 2^rj_2^5}&S^2\vee S^5\ar[u]_{(j_{r+1}^L, 3\beta_{r+1})} \ar[r]^-{\iota_2\vee i_r}& S^2\vee P^{6}(2^r)\ar[u]^{\bar{\omega}}\ar[r]^-{p_{\vee}}& S^6
	} }
	$, ~~~~$ \small{\xymatrix{
			\Omega S^6 \ar@{=}[d]\ar[r]&F_{J_3}^{M}\ar[r] &J_3(M_{2^r\iota_2},\!S^3) \\
			\Omega S^6 \ar[r]&F_{p_{\vee}}\ar[u]^{\omega} \ar[r]& S^2\vee P^{6}(2^r)\ar[u]^{\bar{\omega}}
	} }
	$.
	
	So there is  the following homotopy commutative diagram
	
	$\small{\xymatrix{
			S^2\wedge S^4 \ar[dr]_{j_1^2\wedge \iota_4}\ar[r]^-{j_{r+1}^L\wedge \iota_4}&L_{r+1}^4\wedge S^4\ar[r]^-{\alpha_r} & J_1(M_{\gamma_3},S^5)=L_{r+1}^4\ar[r]^-{I_2}&J_2(M_{\gamma_3},\!S^5)\\
			&(S^2\vee S^5)\wedge S^4\ar[u]_{(j_{r+1}^L, 3\beta_{r+1})\wedge \iota_4} \ar[r]_-{[id, \pm 2^{r}\!j_2^5]~~}& J_1(M_{\pm 2^rj_2^5},S^5)\simeq S^2\vee S^5\ar[u]^{\omega|_{J_1}}\ar[r]& J_2(M_{\pm 2^r\!j_2^5},\!S^5)\ar[u]^{\omega|_{J_2}}
	} }
	$
	\begin{align}
		&\alpha_r(j_{r+1}^L\wedge \iota_4)=(j_{r+1}^L, 3\beta_{r+1})[id, \pm 2^{r}\!j_2^5](j_1^2\wedge \iota_4))~~ (~\text{by}~\omega|_{J_1}=(j_{r+1}^L, 3\beta_{r+1}))\nonumber\\
		&=[(j_{r+1}^L, 3\beta_{r+1}), \pm 3\!\cdot\!2^r\beta_{r+1}]\Sigma(j_1^1\wedge \iota_4)=[(j_{r+1}^L, 3\beta_{r+1})j_1^2, \pm 3\!\cdot\!2^r\beta_{r+1}\iota_5]\nonumber\\
		&=[j_{r+1}^L,  \pm 3\cdot2^r\beta_{r+1}]=\pm 3\!\cdot\!2^r[j_{r+1}^L,  \beta_{r+1}]\nonumber\\
		&=\pm 3\!\cdot\!2^r[\tau_{r+1}^Lj_{r+1}^Fj_1^2,  \tau_{r+1}^Lj_{r+1}^Fj_2^5]=\pm 3\!\cdot\!2^rJ_{r+1}^{LF}[j_1^2,j_2^5].
		\label{ahpha r}
	\end{align}

	(\ref{fiber J3M2r}) induces the following  exact sequence with two commutative squares
	\begin{align}
		\small{\xymatrix{
				\pi_{7}( S^6)\ar[r]^-{(\partial_{J_3}^{M})_{6\ast}}&\pi_{6}(F_{J_3}^{M})\ar[r]&\pi_{6}(J_3(M_{2^r\iota_2},S^3))\ar[r]&\pi_{6}(S^6)\ar[r]^-{ (\partial_{J_3}^{M})_{6\ast}}& \pi_{5}(F_{J_3}^{M})\\
				\pi_{6}(S^5)\ar[r]^-{\gamma_{3\ast}}\ar[u]_{\Sigma \cong } &\pi_{6}(L_{r+1}^4)\ar[u]_{ I_{2\ast}}&&\pi_{5}( S^5)\ar[r]^-{\gamma_{3\ast}}\ar[u]_{\cong\Sigma } &\pi_{5}(L_{r+1}^4)\ar[u]_{I_{2\ast}} }} \label{exact of J3(Mr)}
	\end{align}
	By Lemma \ref{lemma lambda3}, for $r\geq 1$, $\gamma_{3\ast}$ in (\ref{exact of J3(Mr)}) is zero and isomorphic  in the left and right square respectively, then by (\ref{Cofiber Sk7FJ3M}), (\ref{ahpha r}) and Lemma \ref{lem:pi6(Lm4)}
	\begin{align}
		\pi_{6}(J_3(M_{2^r\iota_2},S^3))\!\cong\! \pi_{6}(F_{J_3}^{M})\!\cong\! \frac{\pi_{6}(L_{r+1}^4)}{\left\langle 3\!\cdot\!2^rJ_{r+1}^{LF}[j_1^2,j_2^5] \right \rangle} \!\cong\!
		\left\{\!\! \!
		\begin{array}{ll}
			\Z_2\!\oplus\! \Z_2\!\oplus \!\Z_2\!\oplus\! \Z_2, &\!\! \!\!\hbox{$r=1$;} \\
			\Z_2\oplus\Z_2\!\oplus \!\Z_4\!\oplus\! \Z_{2^{r}} , &\!\! \!\!\hbox{$r\geq 2$.}
		\end{array}
		\right.\nonumber
	\end{align}
\end{proof}

J. Mukai get that $\nu'\in\pi_6(S^3)$ has no lift in $\pi_{6}(P^{3}(2))$  in \cite{Morisugi Mukai lift}. However for $r\geq 2$, the lift in $\pi_{6}(P^{3}(2^r))$ of $\nu'$ exists.

\begin{lemma}\label{lem:lift of nu'}
	If $r\geq 2$, then there is a lift  $\tilde{\nu}_r'\in \pi_{6}(P^{3}(2^r))$ with order $4$ of $\nu'\in\pi_6(S^3)$  under the canonical quotient map $P^{3}(2^r)\xrightarrow{p_r} S^3$.
\end{lemma}
\begin{proof}
	By the (5.4) of \cite{Toda}, the Toda bracket $\{\eta_3,2\iota_4,\eta_4\}_1=\{\nu', -\nu'\}$. From the \cite{Mukai}, there is an element $\tilde{\eta}_3\in \pi_{4}(P^{3}(4))$ with order 2 such that $p_{2}(\tilde{\eta}_3)=\eta_3\in \pi_{4}(S^3) $. $\{\tilde{\eta}_3,2\iota_4,\eta_4\}_1\subset \pi_{6}(P^{3}(4))$ is well defined and $p_{2\ast}\{\tilde{\eta}_3,2\iota_4,\eta_4\}_1\subset \{p_{2}\tilde{\eta}_3,2\iota_4,\eta_4\}_1=\{\eta_3,2\iota_4,\eta_4\}_1=\{\nu', -\nu'\}$. Thus there is an element $\tilde{\nu}_2'\in \{\tilde{\eta}_3,2\iota_4,\eta_4\}_1$, such that $p_{2\ast}(\tilde{\nu}'_2)=\nu'$.
	
	Moreover, $\{\tilde{\eta}_3,2\iota_4,\eta_4\}_1(2\iota_6)\subset \{\tilde{\eta}_3,2\iota_4,\Sigma(\eta_32\iota_4)\}_1=\{\tilde{\eta}_3,2\iota_4,0\}_1\equiv 0 $ mod $\tilde{\eta}_3\Sigma\pi_{5}(S^3)$. Thus the order of $2\tilde{\nu}'_2=\tilde{\nu}_2'(2\iota_6)$ is not lager than 2, hence the order of $\tilde{\nu}'_2$ is 4. Now for $r\geq 3$, let $\tilde{\nu}'_r=\bar{\psi}_{r}^{2}\tilde{\nu}'_2\in \pi_6(P^{3}(2^r))$, where $\bar{\psi}_{r}^{2}$ comes from (\ref{diagram M2s to M2r}) for $s=2$. Then the order $\tilde{\nu}'_r$ is  4 and $p_{r\ast}(\tilde{\nu}'_r)=\nu'$.
\end{proof}

\begin{theorem}\label{lem:pi6(M2r)}
	$\pi_6(P^{3}(2^r))\cong
	\left\{
	\begin{array}{ll}
		\Z_2\oplus \Z_2\oplus \Z_2\oplus \Z_2\oplus \Z_2, & \hbox{$r=1$;} \\
		\Z_2\oplus\Z_2\oplus \Z_4\oplus \Z_4\oplus \Z_{2^{r}} , & \hbox{$r\geq 2$.}
	\end{array}
	\right.
	$
\end{theorem}
From (\ref{diagram M2s to M2r}), (\ref{partial5 on nu'}), for $r\geq 2$, we get the following commutative diagram of exact sequences
\begin{align}
	\small{\xymatrix{
			\pi_7( S^3) \ar@{=}[d]\ar[r]^-{(\partial_1)_{6\ast} }&\pi_6(F_{p_{1}})\ar[d]_{\psi_{r\ast}^1}\ar[r]^{i_{p_1\ast}} & \pi_6(P^{3}(2) )\ar[d]_{\bar{\psi}_{r\ast}^1}\ar[r]^{p_{1\ast}}& \Z_2\{2\nu'\}\ar@{^{(}->}[d]\ar[r]^-{(\partial_1)_{5\ast} }&0\\
			\pi_{7}( S^3 )\ar[r]^-{(\partial_r)_{6\ast} }&\pi_6(F_{p_{r}})\ar[r]^{i_{p_r\ast}}& \pi_6(P^{3}(2^r)) \ar[r]^{p_{r\ast}}& \pi_{6}(S^3) \ar[r]^-{(\partial_r)_{5\ast} }&0
	} } \label{diam p5 M1 to Mr}
\end{align}
Since $\pi_6(F_{p_{1}})\cong \Z_2\oplus\Z_2\oplus\Z_2\oplus\Z_2$,  from Theorem 6.36 of \cite{WJ Proj plane} we get $\pi_6(P^{3}(2))\cong \Z_2\oplus\Z_2\oplus\Z_2\oplus\Z_2\oplus\Z_2$, which implies $(\partial_1)_{6\ast} =0$. So $(\partial_r)_{6\ast} =\psi_{r\ast}^1(\partial_1)_{6\ast} =0$. Thus we have short exact sequence $0\rightarrow\pi_6(F_{p_{r}})\xrightarrow{i_{p_r\ast}} \pi_6(P^{3}(2^r)) \xrightarrow{p_{r\ast}} \pi_{6}(S^3)\xrightarrow{(\partial_r)_{4\ast} }0$. Now this theorem is obtained from Lemma \ref{lem:pi6(J3M2r)} and \ref{lem:lift of nu'}.

\noindent
{\bf Acknowledgement.} The authors would like to thank the reviewer(s)
for  pointing out some mistakes in writing. They also want to  thank Professor Jianzhong Pan  for helpful discussions on the proof of Lemma 3.8. The first author was partially supported by National Natural Science Foundation of China (Grant No. 11701430).

\bibliographystyle{amsplain}

\end{document}